\documentclass[a4paper,12pt]{article}
 \usepackage[english]{babel}
  \usepackage{amsmath,amsfonts,amssymb,amsthm,paralist,nicefrac}
  \usepackage{bbm}

  \usepackage{hyperref} 

 \usepackage[active]{srcltx} 
 \usepackage[dvipsnames]{xcolor}
 \usepackage{color,soul} 

\usepackage{geometry}  
\usepackage[cm]{fullpage}

\newtheorem{theorem}{Theorem}[section]
\newtheorem{corollary}[theorem]{Corollary}
\newtheorem{lemma}[theorem]{Lemma}
\newtheorem{proposition}[theorem]{Proposition}

\theoremstyle{definition}
\newtheorem{definition}[theorem]{Definition}
\newtheorem{remark}{Remark}

\def\N{\mathbb{N}}

\def\R{\mathbb{R}}

\let\e=\varepsilon
\let\vp=\varphi
\let\t=\tilde
\let\ol=\overline

\let\.=\cdot
\let\0=\emptyset
\let\mc=\mathcal
\def\ex{\exists\;}

\def\O{\Omega}
\def\o{\omega}
\def\L{\mathcal{L}}

\DeclareMathOperator*{\esssup}{ess\,sup}
\DeclareMathOperator*{\essinf}{ess{\,}inf}

\newcommand{\lm}[1]{\left\lfloor#1\right\rfloor}

\newcommand{\lgr}[1]{%
    \left\lfloor\hspace{-.33em}\left\lfloor #1\right\rfloor\hspace{-.33em}\right\rfloor%
    }
\newcommand{\ggr}[1]{%
    \left\lceil\hspace{-.33em}\left\lceil #1\right\rceil\hspace{-.33em}\right\rceil%
    }

\newcommand{\su}[2]{\genfrac{}{}{0pt}{}{#1}{#2}}

\def\norma#1{\|#1\|_\infty}

\def\thm#1{Theorem \ref{thm:#1}}
\def\seq#1{(#1_n)_{n\in\N}}
\def\limn{\lim_{n\to+\infty}}

\def\gpe{generalized principal eigenvalue}
\def\MP{maximum principle}

\newenvironment{formula}[1]{\begin{equation}\label{#1}}
                       {\end{equation}\noindent}

\def\Fi#1{\begin{formula}{#1}}
\def\Ff{\end{formula}\noindent}

\setstcolor{red}

\begin{document}
\title{\bf Generalized principal eigenvalues for parabolic operators in bounded domains}

\author{Henri {\sc Berestycki}\,\footnote{(a) Department of Mathematics, University of Maryland, MD 20742, USA,
(b) EHESS, CAMS, 54 boulevard Raspail, F-75006 Paris, France,
 (c) Institute for Advanced Study,  Hong Kong University of Science and Technology, Hong Kong}\ , \
 Grégoire {\sc Nadin}\,\footnote{Institut Denis Poisson, Université d'Orléans, Université de Tours, CNRS,
 Orléans, France}\ , \
 Luca {\sc Rossi}\,\footnote{Istituto ``G.~Castelnuovo'', Sapienza Università di Roma, Rome, Italy}}

\maketitle

\setstcolor{red}



\begin{abstract}
We introduce here new generalized principal eigenvalues for linear parabolic operators with heterogeneous coefficients in space and time. We consider a bounded spatial domain and an unbounded time interval $I$ : $I=\R,\ \R^+$ or $\R^-$, and operators with coefficients having  a fairly general dependence on space and time. 
The notions we introduce rely on the parabolic maximum principle and extend
some earlier definitions introduced for elliptic operators~\cite{BerestyckiHamelRossi, BerNirVaradhan}. 

We first show that these eigenvalues hold the key to understanding the large time behavior and entire solutions of heterogeneous Fisher-KPP type equations. 
We then describe the relation of these quantities with principal Floquet bundles for parabolic operators which provides further characterizations of the principal eigenvalues. These allow us to derive monotonicity properties and comparisons between
generalized principal eigenvalues, as well as perturbation results and further properties involving limit operators. We show that the sign of these eigenvalues encodes 
different versions of the maximum principle for parabolic operators. 
Lastly, we~explicitly compute the generalized principal eigenvalues for several classes of operators 
such as spatial-independent, periodic, almost periodic, uniquely ergodic or random stationary 
ergodic coefficients. 
\end{abstract}

\maketitle

\smallskip



\tableofcontents

\section{State of the art}

We investigate the notion of the
generalized principal eigenvalue for a linear parabolic operator. 
We consider uniformly parabolic operators of the form
$$P=\partial_t-a_{ij}(t,x)\partial_{ij}-b_i(t,x)\partial_i-c(t,x),$$
defined for $t\in I$ and $x\in\O$, where $I$ is an unbounded interval
and $\O$ is a bounded domain of $\R^N$, $N\geq1$.
The partial derivatives $\partial_i$ refer to the spatial variable $x$ only,
and we adopt the summation convention. 
We focus on the Dirichlet boundary condition. Most of the results derived in this paper should hold
true, when properly adapted, for Neumann, Robin or periodic boundary conditions, but we leave such extensions as open problems.


Throughout the paper, we will assume the following hypotheses: 

\begin{equation} \label{hyp:reg}\tag{H} 
\left\{\begin{array}{c}
\O \hbox{ is a bounded Lipschitz domain},\\
\\
a_{ij},  b_{i} , c \in L^\infty(\R\times \O) \hbox{ and } a_{ij}\in \mathcal{C}^0 (\R\times \O) \hbox{ for all } i,j=1,...,N,\\
\\
\exists \alpha>0 \ | \ \forall \xi \in \R^{N}, \quad (t,x)\in \R\times \O, \quad 
a_{ij}(t,x) \xi_{i}\xi_{j}\geq \alpha |\xi|^{2}. \\
\end{array}
\right.
\end{equation}

\bigskip

We choose in this paper to consider operators in non-divergence form. Many of our results could easily be extended to operators in  divergence form. We will sometimes assume that $\nabla a_{ij}\in L^\infty(\R\times \O)$ in order to use previous results  for operators in divergence form (see Section \ref{sec:perturbation} for example).

\subsection{Floquet and Lyapounov exponents for time-periodic and almost periodic operators}

When the coefficients of the operator are time-periodic (say with period $T$), the usual notion of eigenvalue is derived from that of Floquet exponents. That is, we consider the principal eigenvalue $r(P)$ of the Poincar\'e map  $Q_T:u_0(\cdot)\mapsto u(T,\cdot)$, where $u$ is the solution of the Dirichlet problem
\begin{equation}\label{Pu=0+} 
    \begin{cases}
	 Pu=0& \text{in }(0,+\infty)\times\O\\
	 u=0& \text{on }(0,+\infty)\times\partial\O,
    \end{cases}
\end{equation}
emerging from the initial datum $u_0$.
It is easily seen that one can write $r(P)=e^{-\lambda (P)T}$, where $\lambda (P)$ is the principal eigenvalue of the operator $P$, that is, the unique $\lambda$ associated with some positive function $\phi$ on $(0,T)\times \O$, with $\phi=0$ on $(0,T)\times \partial \O$, $\phi(\cdot +T,\cdot)\equiv \phi$, and $P\phi=\lambda \phi$ on $(0,T)\times \O$. 
Its existence, together with the simplicity, follows from the Krein-Rutman theory.

A comprehensive study of these eigenvalue has been provided by Hess \cite{Hess}. We refer to \cite{Nadinper} for some extensions to space-periodic coefficients. 

For almost periodic or random stationary ergodic coefficients with respect to $t$, the Poincar\'e map~$Q_T$ is not well-defined and the inverse of the operator $P$ is not compact anymore. 
This prevents the application of the Krein-Rutman theory to obtain
 the principal eigenvalue. However, one can still prove the existence of a finite quantity $\lambda$ which plays the role of the Floquet exponent, i.e.~such that any solution of problem \eqref{Pu=0+} emerging from a positive initial datum satisfies
$$\lim_{t\to +\infty}\frac{1}{t}\ln \|u(t,\cdot)\|_{L^\infty (\O)}=-\lambda.$$
This indeed means that positive solutions behave like
$e^{-\lambda t}$ for large $t$.
The above limit holds with some uniformity with respect to translations of the coefficients 
in almost periodic media~\cite{SYbook}, and for almost every event in the case of random stationary ergodic coefficients \cite{MSbook}. Indeed, the $L^\infty$ norm could be replaced by other norms or pointwise convergence using classical parabolic estimates. In the almost periodic setting, the quantity~$\lambda$ is called a Lyapounov exponent. 
Moreover, for sign-changing initial data, one can prove some exponential separation (see \cite{MSbook, SYbook}), extending earlier spectral gap results that were derived in the periodic framework. 

It is easily seen that the above convergence property 
cannot hold in general when the dependence with respect to~$t$ of the coefficients is arbitrary. Among other things, we will show in this paper that for operators
with general time/space heterogeneity of the coefficients, the semi limits $$\limsup_{t\to +\infty}\frac{1}{t}\ln \|u(t,\cdot)\|_{L^\infty (\O)}
\quad\text{ and }\quad
\liminf_{t\to +\infty}\frac{1}{t}\ln \|u(t,\cdot)\|_{L^\infty (\O)},$$ 
are characterized by two distinct notions of generalized principal eigenvalue,
cf.\ Theorem~\ref{thm:caraceta} below.

\subsection{Generalized principal eigenvalues for elliptic operators}

The notion of generalized principal eigenvalues for elliptic operators has first been introduced by Berestycki, Nirenberg and Varadhan \cite{BerNirVaradhan} 
in the case of bounded, non-smooth domains
(see also Agmon~\cite{A1} and Nussbaum and Pinchover \cite{NP} for an earlier related notion)
and then extended by Berestycki, Hamel and 
Rossi~\cite{BerestyckiHamelRossi}
to unbounded domains. Consider an elliptic operator 
\Fi{L}
\mathcal{L}=a_{ij}(x)\partial_{ij}+b_i(x)\partial_i+c(x), \qquad x\in\O
\Ff
(with the summation convention on $i,j$).
The  notion of generalized principal eigenvalue of elliptic operators of \cite{BerNirVaradhan} reads
$$\lambda(-\mathcal{L}):=\sup\{\lambda\ :\ \ex\phi>0 \text{ in }\O,\ -\mathcal{L}\phi\geq\lambda\phi\text{ in }\O\}.$$
This notion extends the relevant properties of the classical
principal eigenvalue to the case where $\O$ is non-smooth.
This is no longer the case when the domain $\O$ is unbounded.
In order to deal with the unbounded case, an alternative notion was introduced in \cite{BHR1,BerestyckiHamelRossi}:
$$\lambda'(-\mathcal{L}):=\inf\{\lambda\ :\ \ex\phi>0 \text{ in }\O,
\ \phi=0 \text{ on }\partial\O,
\ -\mathcal{L}\phi\leq\lambda\phi\text{ in }\O\},$$
and still another one in \cite{BR4}:
$$\lambda''(-\mathcal{L}):=\sup\{\lambda\ :\ \ex\phi, \ \inf_\O \phi>0 \text{ in }\O,\ -\mathcal{L}\phi\geq\lambda\phi\text{ in }\O\}.$$
These notions incorporate in their definition the Dirichlet boundary condition.
For their extension to Neumann or oblique derivative boundary conditions we refer to~\cite{R-oblique}.

It has been proved in \cite{BerestyckiHamelRossi, BR4} that, for an arbitrary
elliptic operator, there holds $$\lambda''(-\mathcal{L})\leq \lambda'(-\mathcal{L})\leq \lambda(-\mathcal{L}).$$
It may happen that $\lambda'(-\mathcal{L})<\lambda(-\mathcal{L})$. But equality holds if 
the operator is self-adjoint. 
Further relevant cases of equalities between these quantities are established in
\cite[Theorem 1.9]{BR4}.

These eigenvalues provide both necessary and sufficient conditions for the validity
of the maximum principle. They also play a central role in 
the study of heterogeneous Fisher-KPP type equations, as they allow to characterize existence, positivity, and attractivity of the stationary solutions, see \cite{BerestyckiHamelRossi, BR4} and the discussion
in the next section. 

In the present paper, we want to find a notion of 
principal eigenvalue that fits for a parabolic operator $P$ on a bounded, regular, spatial domain $\O$.
Here we face the same issue of lack of compactness
as in the elliptic case when $\O$ is unbounded, because we
consider unbounded intervals of times.
As in the elliptic case, this leads us to introduce several notions of generalized principal eigenvalue.


\subsection{A first notion of generalized principal eigenvalue}\label{sec:hyp}

Let us start by introducing a first notion of 
{\em generalized principal eigenvalue} for the parabolic operator $P$ on a time interval $I$ that can be one of the following:
$$\R,\qquad \R^+:=(0,+\infty),\qquad \R^-:=(-\infty,0).$$
We consider a fixed, open ball $B\Subset \O$, i.e.\ such that $\bar B\subset \O$.
Then we define
\begin{align*}
\mu_{b,p}(I):=\inf\{&\lambda\ :\ \ex\phi>0\text{ in }I\times\O, \ \phi=0\text{ on }I\times\partial\O,\\
&\ \inf_{I\times B}\phi>0,\ \sup_{I\times\O}\phi<+\infty,\ 
P\phi\leq\lambda\phi\text{ in }I\times\O\},
\end{align*}

The indices ``$b,p$'' refer to the requirements that the ``test-functions'' $\phi$
are bounded and strictly positive.
The functions $\phi$ are also assumed to belong to the Sobolev space 
$W^{1,2}_{p,\,\mathrm{loc}}(I\times\ol\O)$ for
some $p>N+1$, that is, 
$\phi, \partial_{t}\phi, \partial_i \phi, \partial_{ij}\phi \in L^{p}(J\times \O)$
for any compact interval $J\subset I$,
and the differential inequalities are understood to hold in the
strong sense, i.e.\ almost everywhere.
As a consequence, the functions $\phi$ are continuous up to the boundary of $\O$, by Morrey's embedding.



It is a priori not obvious that the the generalized principal eigenvalue 
$\mu_{b,p}(I)$ is well defined and finite.
This is a fact that we show as a byproduct of our main results in Proposition~\ref{pro:basic} below.
We will further see that $\mu_{b,p}(I)$
does not depend on the choice of the ball $B$. 
Actually, its definition would not change if one requires 
the stronger condition 
$\inf_{I\times K}\phi>0$ for every compact set $K\subset\O$,
or the weaker one
$\inf_{t\in I}\|\phi(t,\.)\|_{L^\infty(\O)}>0$
(see Remark \ref{rk:ball}$(\ref{rk:ballB})$ below). 

It is natural to wonder whether there exists a 
{\em generalized principal eigenfunction} associated 
with $\mu_{b,p}(I)$, i.e., a positive eigenfunction fulfilling the requirements 
$\inf_{I\times B}\phi>0$, $\sup_{I\times\O}\phi<+\infty$.
We will see that in general this is not the case (see \ Remark \ref{rmk:differ} below).

In Section \ref{sec:other}, we discuss other possible definitions of generalized eigenvalues playing with the requirements imposed on the test functions or with sup rather than inf. We will see that they are relevant in particular
for the study of linear Dirichlet problems.



\subsection{Applications to semilinear problems}\label{sec:semilinear}

In order to illustrate the use of the new notion $\mu_{b,p}$
of generalized principal eigenvalue, we now state some 
applications to the study of  the semilinear parabolic~problem
\Fi{RD}
\begin{cases}
\partial_t u-a_{ij}(t,x)\partial_{ij} u-b_i(t,x)\partial_i u= f(t,x,u), &
t\in I,\ x\in\O,\\
u(t,x)=0,  & t\in I,\ x\in\partial \O,
\end{cases}
\Ff
where $I$ will be either $\R^\pm$ or $\R$.
We will always assume here that the nonlinear term $f=f(t,x,s)$ is of class 
$\mathcal{C}^{1}$ with respect to $s$, uniformly with respect to 
$(t,x) \in I\times \O$, 
with bounded derivative $f_s'$.
We further assume that 
\Fi{f=0}
\forall(t,x)\in I\times\O,\quad
f(t,x,0)=0,
\Ff
and that there exists a constant $M>0$ such that
\Fi{f<0}
\forall(t,x)\in I\times\O, \ \forall s\geq M,\quad
f(t,x,s)\leq 0.
\Ff
Lastly, we will consider the KPP hypothesis:
\begin{equation}\label{hyp:fKPP} 
\forall s>0,\ x\in\O, \quad \inf_{t\in I}
\big(f_s'(t,x,0)s-f(t,x,s)\big)>0.\end{equation}


For the Fisher KPP equation without temporal dependence on a bounded, non-smooth domain~$\O$, 
\cite[Theorems 1.1 and 2.2]{BerNirVaradhan} imply that there exists a positive 
stationary solution if and only if the eigenvalue $\lambda(-\L)$ of
the linearized elliptic operator $\L$ is negative.
This fact is extended to unbounded smooth domains in \cite{BerestyckiHamelRossi}, using also the notion $\lambda'(-\L)$.
The eigenvalue $\lambda(-\L)$ also allows one to derive conditions for the uniqueness  and 
the stability of this stationary state.

Keeping in mind these results, we aim to derive here analogous 
properties for the reaction-diffusion problem \eqref{RD} 
on a bounded
domain $\O$, under general space/time dependence of the coefficients,
based on criteria involving generalized principal eigenvalues. 

A first result concerns
the long-time behavior of solutions to the Cauchy problem. It turns out that the quantity $\mu_{b,p}(\R^+)$
encodes the {\em persistence property}, that is, 
the fact that solutions preserve a positive mass as $t\to+\infty$. This information is particularly relevant in the context of population dynamics models.

\begin{proposition}\label{prop:initdatum}
Assume that $I=\R^+$, that $\O$ is of class $\mathcal{C}^{2,\alpha}$, for some $0<\alpha<1$, 
and that $f$ satisfies \eqref{f=0}-\eqref{hyp:fKPP}.
Let $u$ be the solution of problem~\eqref{RD} 
with a continuous, nonnegative initial datum $u_{0}\not\equiv 0$.
Let $\mu_{b,p}(\R^+)$ be the eigenvalue associated with
the linearized operator around $0$, namely
\Fi{opP}
P:=\partial_t -a_{ij}(t,x)\partial_{ij} -b_i(t,x)\partial_i - f_{s}'(t,x,0).
\Ff
Then the {\em persistence property} 
\begin{equation}\label{eq:persistence}
\liminf_{t\to +\infty}u(t,x)>0 \quad \hbox{ for all } x\in \O,
\end{equation}
holds if and only if $\mu_{b,p}(\R^+)<0$. 
\end{proposition}


Next, we  focus on entire (i.e.\ time-global) solutions to
the semilinear Dirichlet problem.
Namely, we consider~\eqref{RD} with $I=\R$.
We are concerned with solutions that are uniformly positive away from the boundary of $\O$, that is,
\Fi{infK}
\forall K\Subset\O,\quad\inf_{I\times K}u>0.
\Ff
The following result shows that the existence of these solutions
is completely characterized by the sign of~$\mu_{b,p}(\R)$, as in the temporal-independent case.

\begin{proposition}\label{pro:RD-existence} 
	Assume that $I=\R$, that $\O$ is of class $\mathcal{C}^{2,\alpha}$, for some $0<\alpha<1$,
    and that $f$ satisfies \eqref{f=0}-\eqref{hyp:fKPP}.
 Let $\mu_{b,p}(\R)$ be the eigenvalue associated with
the linearized operator $P$ given by \eqref{opP}.
Then the problem \eqref{RD}
admits a bounded entire solution satisfying \eqref{infK} if and only if $\mu_{b,p}(\R)<0$.
\end{proposition}

The two results above can be deduced from the results of \cite{MSKol} 
under the additional hypothesis of random stationarity and ergodicity of the coefficients with respect to time. We will describe this hypothesis in Section \ref{sec:cases}. In the present article, we prove these results in full generality.

Let us point out that the strict inequality in \eqref{hyp:fKPP}, as well as
the $\mathcal{C}^{2,\alpha}$ regularity of $\O$, are only required in Propositions 
\ref{prop:initdatum}-\ref{pro:RD-existence}
to handle the case $\mu_{b,p}=0$.
While the former is really necessary, we believe that the latter could be relaxed.

Next, we show that if one assumes the following stronger version of the KPP hypothesis \eqref{hyp:fKPP}:
\Fi{fconcave}
\forall 0<s<s',\quad \inf_{\su{t\in I}{x\in\O}}\left(\frac{f(t,x,s)}{s}-\frac{f(t,x,s')}{s'}\right)>0,
\Ff
then the bounded solution to problem~\eqref{RD},~\eqref{infK} is unique.
This is true not only for entire solutions ($I=\R$), but also for {\em ancient} solutions ($I=\R^-$).

\begin{theorem}\label{thm:RD-uniqueness}
	Assume that $I=\R^-$, that the $a_{ij}$ are uniformly continuous on $\R^-\times\O$
 and that $f$ satisfies \eqref{f=0}, \eqref{f<0}, \eqref{fconcave}.
Then problem \eqref{RD}, \eqref{infK}
admits at most a unique bounded ancient solution.
\end{theorem}

Theorem \ref{thm:RD-uniqueness} 
extends a uniqueness result for ancient solutions obtained by 
Rodr\'iguez-Bernal and Vidal-L\'opez \cite{RB-VL} in the particular case where 
$a_{ij}(t,x)\equiv\delta_{ij}$ and $b_i(t,x)\equiv 0$, and assuming that the solutions are ordered.
Namely, \cite[Theorem 2]{RB-VL} asserts that 
if $u_{1}$ and $u_{2}$ are two bounded ancient solutions of \eqref{RD}, \eqref{infK}
such that $u_1\leq u_{2}$, then $u_{1}\equiv u_{2}$. 


\subsection{Aims of this study}

The above results highlight the relevance of the generalized principal eigenvalue $\mu_{b,p}$ for discussing the evolution Fisher-KPP type equation (\ref{RD}). 
The purpose of this paper is to introduce various possible notions of generalized principal eigenvalues and to investigate their properties. We have three main objectives. 

Our first task is to extend the notion of generalized principal eigenvalues from the elliptic case to the parabolic framework. Furthermore, following the systematic study carried out for elliptic operators, we would like to derive comparison or equivalence properties between the various notions of generalized principal eigenvalues in the parabolic setting. 

Secondly, we establish a connection between the generalized principal eigenvalues
and the Floquet theory. In doing so, we derive a new characterization of the exponential growth of the Floquet bundles when the coefficients of the parabolic operator have a general dependence on $t$. For some classes of operators (such as operators with uniquely ergodic or random stationary ergodic coefficients), we are then able to establish the equivalence between the various notions of generalized principal eigenvalues. 

Finally, we show some applications of the generalized principal
eigenvalues to the qualitative study of solutions
of linear and semilinear parabolic problems.

All these tasks will also give a unified point of view on
some results obtained in the literature using distinct tools and techniques.

The main technical tool employed in our proofs
is a new notion of average for possibly unbounded functions,
that we call {\em global growth-rate}, 
see Definition~\ref{def:lgr} below, which extends a previous notion introduced by the last two authors of the present work in~\cite{NR1}.

\medskip
We hope that this work will inspire future investigations aiming at extending the theory of parabolic generalized principal eigenvalues to unbounded domains. In such domains, the Floquet approach is not well-established. For want of such a theory, most of the results we discuss here are open in general in unbounded domains. Since generalized principal eigenvalues have proved to be successful in the discussion of problems in unbounded domains in the elliptic case, it is natural to expect that the same will be true in the parabolic setting as well.

\medskip

\textbf{Organization of the paper.}
In Section~\ref{sec:other} we introduce several definitions of generalized principal eigenvalue.
Next, we reclaim  the notion of principal Floquet bundle and we state the characterization of the generalized principal eigenvalues in terms of such notion in Section~\ref{sec:caracHuska}.
These results will allow us to derive some comparisons between the
different generalized principal eigenvalues, 
which are stated in Section~\ref{sec:comp}.
The  two following subsections
contain applications to the qualitative study of solutions
of linear problems and the maximum principle for ancient solutions.
In Sections~\ref{sec:perturbation},~\ref{sec:cases} and \ref{sec:*}
we summarize and re-frame some known results in the light
of our generalized principal eigenvalues:
perturbation results, examples and
computation in some particular cases and for limit operators, 
relation with the notion of exponential type.
Section~\ref{sec:tool}
contains some technical results: the H\"older regularity of the Floquet 
bundle, and the properties of our main tool, the global growth-rate,
which is of independent interest.
The last two sections contain the proofs of the connection 
with the Floquet bundle and of the remaining main results.

\medskip

\textbf{Acknowledgement.} This research has received funding from the French ANR Project ANR-23-CE40-0023-01 ReaCh.
LR was supported by the 
European Union -- Next Generation EU, on the PRIN project
 ``PDEs and optimal control methods in mean field games, population dynamics and multi-agent models'', and by INdAM Gnampa.


\section{Definitions and properties of the generalized principal eigenvalues
} 

\subsection{Further definitions of generalized principal eigenvalues}
\label{sec:other}

We now define further quantities that we also call {\em generalized principal eigenvalues}, by modifying the requirements
on the ``test-functions'' on which the operator $P$ acts. 
Such requirements are reflected by the symbol used to indicate the generalized principal eigenvalue,
according to the following notation table:


\begin{center}
\begin{tabular}{|c|l|}
        \hline
	$\lambda$ \  & \ supersolutions \\
        $\mu$     \ & \ subsolutions \\
	  ${(\;)}_b$ \ & \ bounded \\
        ${(\;)}_p$     \ & \ locally strictly positive \phantom{$A_{A_{A_{A_a}}}$}\\
        \hline
\end{tabular}
\end{center}

\noindent 
More precisely, we consider a fixed, open ball $B\Subset \O$, i.e.\ such that $\bar B\subset \O$,
and we define 
$$\lambda_{b,p}(I):=\sup\{\lambda\ :\ \ex\phi>0, \ \sup_{I\times\O}\phi<+\infty,\
\inf_{I\times B}\phi>0,\
P\phi\geq\lambda\phi\text{ in }I\times\O\}.$$

$$\mu_b(I):=\inf\{\lambda\ :\ \ex\phi>0 \text{ in }I\times\O,\  
	\sup_{I\times\O}\phi<+\infty,\ \phi=0\text{ on }I\times\partial\O,\
	P\phi\leq\lambda\phi\text{ in }I\times\O\}$$

$$\lambda_p(I):=\sup\{\lambda\ :\ \ex\phi>0 \text{ in }I\times\O,\ 
	\inf_{I\times B}\phi>0, \ P\phi\geq\lambda\phi\text{ in }I\times\O\}$$

%

$$
\mu_{p}(I)
:= \inf \{\lambda, \ \exists \phi>0, \
\inf_{I\times B}\phi > 0, \ \phi=0 \hbox{ on } I\times \partial \O,\ 
P\phi \leq \lambda \phi \hbox{ in } I\times\O \}$$

$$
\lambda_{b}(I)
:= \sup \{\lambda, \ \exists \phi >0, \ \sup_{I\times\O}\phi<+\infty,\ P\phi \geq \lambda \phi \hbox{ in }
I\times\O\}\,.$$

As in the definitions of $\mu_{b,p}$, 
the test-functions $\phi$ are assumed to belong to  
$W^{1,2}_{p,\, \mathrm{loc}}(I\times\ol\O)$ for some $p>N+1$.
The quantities $\lambda,\mu$ are defined as suprema and infima respectively, and doing it the other way around would yield $-\infty$ and $+\infty$ respectively.

We will show that the notions of generalized principal eigenvalue
introduced so far, including $\mu_{b,p}$,
coincide in some cases (see Propositions \ref{pro:elliptic}, \ref{pro:periodic},
\ref{pro:ap}, \ref{pro:ergo} below),
but in general they differ, see Remark \ref{rmk:differ}.
Some inequalities between them, together with some further properties, are 
provided by Proposition \ref{pro:basic} below.


\subsection{The relation between generalized principal eigenvalues and the Floquet bundle}\label{sec:caracHuska}


We consider here positive entire -- i.e., global-in-time -- solutions of the Dirichlet
problem in a bounded Lipschitz domain $\O$, that is,
\Fi{Pu=0}
	 \begin{cases}
	 Pu=0& \text{in }\R\times\O\\
	 u=0& \text{on }\R\times\partial\O.
	 \end{cases}
\Ff
The existence and uniqueness (up to normalization) of the positive entire solution 
has been derived in several frameworks \cite{CLM-Floquet,Huska, M, P, PT} and is a first step to extend
the Floquet theory to non-periodic settings. In the present paper, we will mostly rely on the work of H{\'u}ska, Pol{\'a}{\v{c}}ik and 
Safonov~\cite{Huska}.

\begin{theorem}[{\cite[Theorem 2.6]{Huska}}]\label{thm:uP}
There exists a unique time-global, nonnegative solution $u_{P}$
 to~\eqref{Pu=0} normalized by $\|u_{P}(0, \cdot)\|_{L^{\infty}(\O)}=1$. 
\end{theorem}

The collection of the one-dimensional spaces $X_1(t):=\{k u_P (t,\cdot)\ :\ k\in\R\}$, $t \in \R$, is called the 
{\em principal Floquet bundle} of (\ref{Pu=0}).
%


If the operator $P$ is just defined on the domain $\R^+\times\O$ or on $\R^-\times\O$,
then we extend it by even reflection with respect to $t$, and we let $u_P$ denote the
corresponding function given by Theorem~\ref{thm:uP}.
As a matter of fact, despite of the fact that the particular choice of extension of the operator $P$ 
clearly affects the function 
$u_{P}$, our results will be independent of it. They hold true 
for any arbitrary extension, as long as hypotheses \eqref{hyp:reg} are fulfilled,
(see Remark \ref{rk:ball}$(\ref{rk:extension})$  below).

The function $u_{P}$ is continuous, in a certain sense, with respect to $L^{\infty}$ perturbations of the coefficients 
of $P$ and of the domain $\O$, see  \cite[Theorems 2.8 and 2.9]{Huska}. 

It is also proved in  \cite[Theorem 2.6]{Huska} (see also \cite{M, Map, P, PT}) that, for general solutions, exponential separation occurs, meaning that
the long time behavior is essentially determined by the projection on the space generated by $u_P$, up to a factor converging exponentially to $0$ in time, as stated in the next theorem.

\begin{theorem}[\cite{Huska}] \label{thm:Huska}
There exist $C,\gamma>0$ such that, 
for any continuous initial datum $u_0$ vanishing on $\partial\O$,
the solution to the problem $Pu=0$ on $(0,+\infty)\times\O$ with Dirichlet boundary condition satisfies
$$\forall t>0,\quad
\|u(t,\.)-q u_{P}(t,\.)\|_{L^{\infty}(\O)}\leq 
C\|u_{0}-q u_{P}(0,\.)\|_{L^{\infty}(\O)}
\|u_{P}(t,\cdot)\|_{L^{\infty}(\O)}\,e^{-\gamma t}\,,$$
for some $q\in \R$. Moreover, $u(t,\.)-q u_{P}(t,\.)$
has a zero in $\O$ for all $t \geq 0$.
\end{theorem}

This result is an immediate application of \cite[Theorem~2.6]{Huska}. Indeed, 
statement \cite[Theorem~2.6(ii)]{Huska} guarantees the existence of a unique $q\in\R$ such that $u_0-q u_P(0,\.)$ 
belongs to the {\em complementary Floquet bundle} associated to the problem,
meaning that $u(t,\.)-q u_P(t,\.)$ vanishes somewhere in $\O$, for any $t>0$.
Thus statement \cite[Theorem~2.6(iii)]{Huska} yields Theorem \ref{thm:Huska}.

When the operator is in divergence form, that is, 
$$P u:=\partial_t u-\partial_{i}\big(a_{ij}(t,x)\partial_{j}u\big)-\tilde{b}_i(t,x)\partial_i u-c(t,x)u,$$
the quantity $q$ is actually explicit. Namely, in this case we have
$q=\int_\O u_0(x)u_{P^\star}(0,x)dx$, where $P^\star$ is the adjoint operator:
$$P^\star u:=-\partial_t u-\partial_{i}\big(a_{ij}(t,x)\partial_{j}u\big)+\partial_i\big( \tilde{b}_i(t,x)u\big)-c(t,x)u,$$
and $u_{P^\star}$ is the unique entire solution of $P^\star u_{P^\star}=0$ on $\R\times \O$, $u_{P^\star}=0$ on $\R\times \partial\O$, normalized by $\int_\O u_P(0,x)u_{P^\star}(0,x)dx=1$. The function $u_{P^\star}$ is well-defined thanks to Theorem \ref{thm:Huska}, with the change of variables $t\mapsto -t$. 

\bigskip

We are now in position to state the results providing the link between Floquet bundles
and the generalized principal eigenvalues. 


\begin{theorem}\label{thm:caracmu}
For $I= \R,\;\R^+$ or $\R^-$, there holds
$$\mu_{b,p}(I)=-\lim_{t\to +\infty}
\left(\inf_{s,s+t\in I}\frac{\ln \|u_{P} (s+t,\cdot)\|_{L^{\infty}(\O)} - 
	\ln \| u_{P} (s,\cdot)\|_{L^\infty(\O)}}{t}\right),$$
$$\lambda_{b,p}(I)=-\lim_{t\to +\infty}
\left(\sup_{s,s+t\in I}\frac{\ln \|u_{P} (s+t,\cdot)\|_{L^{\infty}(\O)} - 
	\ln \| u_{P} (s,\cdot)\|_{L^\infty(\O)}}{t}\right).$$
\end{theorem}

In the Floquet bundles theory, the closed interval $[-\mu_{b,p}(I), -\lambda_{b,p}(I)]$ is called the {\em principal spectrum} of the operator $P$ (see for example \cite{HST, Map, MS, MSbook}).

We remark that the $L^\infty$ norm in Theorem \ref{thm:caracmu}
can be replaced by any $L^p$ norm, $p\geq1$, owing to the Harnack inequality,
cf.\ \cite[Corollary 5.3 $(ii)$]{Huska}.
A first consequence is the following. 
\begin{corollary}\label{cor:caracmu}
There holds
$$\mu_{b,p}(\R) = \max \{\mu_{b,p}(\R^-), \mu_{b,p}(\R^+)\},$$
$$\lambda_{b,p}(\R) = \min \{\lambda_{b,p}(\R^-), \lambda_{b,p}(\R^+)\}.$$
\end{corollary}
We now state further characterizations and comparison results.
\begin{theorem}\label{thm:caraceta}
There holds
$$\mu_p(\R)=\mu_p(\R^+)= - \liminf_{t\to +\infty}\frac{\ln \|u_{P} (t,\cdot)\|_{L^{\infty}(\O)}}{t},$$
$$\lambda_b(\R)=\lambda_b(\R^+)= -\limsup_{t\to +\infty}\frac{\ln \|u_{P} (t,\cdot)\|_{L^{\infty}(\O)}}{t}.$$
In particular, one has $\lambda_b(\R)\leq\mu_p(\R)$.

\end{theorem}

\begin{theorem}\label{thm:caraclambda}
There holds
$$\mu_b(\R)=\mu_b(\R^-) =- \liminf_{t\to  -\infty}\frac{\ln \|u_{P} (t,\cdot)\|_{L^{\infty}(\O)}}{t},$$
$$\lambda_p(\R)=\lambda_p(\R^-)= -\limsup_{t\to -\infty}\frac{\ln \|u_{P} (t,\cdot)\|_{L^{\infty}(\O)}}{t}.$$
In particular,  one has     $\lambda_p(\R)\leq\mu_b(\R)$.

\end{theorem}

\begin{remark}\label{rk:ball} 
\begin{enumerate}[$(i)$]

\item\label{rk:extension} 
Since the \gpe s associated with the intervals 
$\R^-$ or $\R^+$ only depend on the operator $P$ restricted to such intervals, 
one deduces that the quantities at the the right-hand sides of the formulas 
in Theorems \ref{thm:caracmu},
\ref{thm:caraceta}, \ref{thm:caraclambda} do not 
depend on the values of the coefficients of $P$ outside the interval $I$, or outside 
$\R^+$ or $\R^-$ respectively.
In particular, those theorems hold true independently of 
the extension of $P$ used in order to 
define the function $u_P$ when $I\neq\R$, 
provided it preserves the hypotheses \eqref{hyp:reg}.

\item\label{rk:IT} As a byproduct of Theorems \ref{thm:caracmu},
\ref{thm:caraceta}, \ref{thm:caraclambda} one infers that the quantities
$\mu_{b,p}(I)$,  $\lambda_{b,p}(I)$, $\mu_p(I)$, $\lambda_b(I)$,
$\mu_b(I)$, $\lambda_p(I)$ do not change if one 
replaces the interval $I=\R^-$ or 
$I=\R^+$ with any half-line $(-\infty,T)$ or $(T,+\infty)$ respectively.

\item\label{rk:ballB} As a consequence of the above results,
it turns out that the definition of the eigenvalues
$\mu_{b,p}$, $\lambda_{b,p}$, $\mu_{p}$, $\lambda_{p}$ do not depend on the 
choice of the ball $B\Subset \O$ in the condition $\inf_{I\times B}\phi>0$
required on the test functions.
We will actually show that one could equivalently require 
the stronger condition
$\inf_{I\times K}\phi>0$ for every compact set $K\subset\O$,
or the weaker one
$$\inf_{t\in I}\|\phi(t,\.)\|_{L^\infty(\O)}>0,$$
(and even a weaker condition) without changing the definitions, 
see Remarks \ref{rk:relax_bp}, \ref{rk:relax} below.
\end{enumerate}
\end{remark}


\subsection{Comparison between the different generalized principal eigenvalues}
\label{sec:comp}

The characterization of the generalized 
principal eigenvalues via the Floquet bundle
yield some relations between the different 
generalized principal eigenvalues.
We summarize them in the following statement. 
It also includes a monotonicity property with respect to the inclusions  of the domains.
\begin{proposition} \label{pro:basic} 
For $I= \R,\;\R^+$ or $\R^-$, 
the quantities $\lambda_{b,p}(I)$, $\mu_{b,p}(I)$, $\mu_p(I)$, $\lambda_b(I)$,
$\mu_b(I)$, $\lambda_p(I)$
are well defined and finite, except for $\lambda_b(\R^-)=\lambda_p(\R^+)=+\infty$,
$\mu_p(\R^-)=\mu_b(\R^+)=-\infty$, 
and satisfy
\[
\begin{array}{c}
\displaystyle
-\sup_{I\times\O}c\;\leq\;\lambda_{b,p}(I)\;\leq\; \min\{\lambda_b(I),\lambda_p(I)\}
\;\leq\;
\max\{\mu_b(I),\mu_p(I)\}\;\leq\; \mu_{b,p}(I).
\end{array}
\]
Moreover, these eigenvalues are nonincreasing with respect to the inclusion of the domain $\O$,
in the sense that the ones associated with the domain $\O$ are
smaller than or equal to the corresponding ones associated with a
smooth domain $\O'\subset\O$.
\end{proposition}


\subsection{Application to the study of the linear Dirichlet  problem}

We consider here the Dirichlet problem
\eqref{Pu=0} for $t>0$, complemented with an initial datum
$u_{0}$ satisfying
\Fi{u0}
u_{0}\in\mc{C}(\ol\O),\qquad u_0\geq0,\qquad u_0\not\equiv 0,\qquad
u_0=0\text{ on }\partial\O.
\Ff
The unique solution 
of this problem is denoted by $u(t,x;u_{0})$.
The long-time behavior of this solution can be described combining 
the exponential separation of \cite{Huska} with our characterization 
Theorem~\ref{thm:caraceta}.

\begin{proposition} \label{prop:carac} 
Let $u_{0}$ be an initial datum satisfying \eqref{u0}.
Then, for any $x\in\O$ the following~hold:
$$
\liminf_{t\to +\infty}\frac{\ln u(t,x;u_{0})}{t}=-\mu_{p}(\R^+),$$
$$\limsup_{t\to +\infty}\frac{\ln u(t,x;u_{0})}{t} =-\lambda_b(\R^+).$$
\end{proposition}


Next, for all $s\in\R$, we let $u^s=u^s(t,x;u_{0})$ denote the solution of $Pu^s=0$ in $(s,+\infty)\times \O$ with initial datum $u^s(s,x;u_{0})=u_{0}(x)$ for all $x\in \O$. 

\begin{proposition}\label{prop:u^s}
Let $u_{0}$ be an initial datum satisfying \eqref{u0}.
Then, for any $x\in\O$ the following~hold:
$$\liminf_{s\to -\infty}\frac{\ln u^s(0,x;u_{0})}{s}=\lambda_p(\R^-),$$
$$\limsup_{s\to -\infty}\frac{\ln u^s(0,x;u_{0})}{s}=\mu_b(\R^-).$$
\end{proposition}

 \subsection{Application to the \MP\ for ancient solutions}\label{sec:MP}

We now turn to the study of ancient solutions to the Dirichlet problem \eqref{Pu=0}.

\begin{definition}\label{def:MP}
	We say that the operator $P$ satisfies the {\em\MP\,}
	({\em MP }for short) on $\R^-\times\O$ if for every ancient subsolution
        $u$ to \eqref{Pu=0}, i.e.\ satisfying
	$$Pu\leq0 \ \text{ in }\R^-\times\O,\qquad
	\sup_{\R^-\times\O}u<+\infty,\qquad
	u\leq0 \ \text{ on }\R^-\times\partial\O,$$
	there holds $u\leq0$ in $\R^-\times\O$.
\end{definition}

\begin{theorem}\label{thm:MP}
	The operator $P$ satisfies the MP on $\R^-\times\O$
  if $\mu_b(\R^-)>0$,
  and only if $\mu_b(\R^-)\geq 0$.
\end{theorem}

\begin{remark}
    Since $\lambda_p(\R^-)\leq\mu_b(\R^-)$, thanks to Theorem \ref{thm:caraclambda},
    a sufficient condition for the {\em MP} to hold is $\lambda_p(\R^-)>0$.
    This is actually easier to check than $\mu_b(\R^-)>0$,
    because $\lambda$ is defined as a supremum rather than an infimum.

    We also point out that in the case $\mu_b(\R^-)=0$, the {\em MP} may or may not hold.
    Indeed, on one hand, for $P=\partial_t-\L$ with $-\L$ elliptic operator
    satisfying $\lambda_{D}(-\L,\O)=0$, 
    Proposition~\ref{pro:elliptic} yields $\mu_b(\R^-)=\lambda_{D}(-\L,\O)=0$,
    but the Dirichlet principal eigenfunction $\phi_{D}$ violates the {\em MP}.
    
    On the other hand, adding to the operator 
    a term $\sigma(t):=-(1-t)^{-1/2}$, that is, 
    considering $P=\partial_t-\L+\sigma(t)$, one still has 
    $\mu_b(\R^-)=\lambda_{D}(-\L,\O)=0$, owing to Proposition \ref{prop:indet}, but
    $$u_{P}(t,x)= \phi_{D}(x) e^{-\lambda_{D}(-\L,\O) t+\int_{0}^{t}\sigma(s)ds}$$ 
    tends to $+\infty$ as $t\to-\infty$, locally uniformly in $x\in\O$.
    As seen in the proof of Theorem \ref{thm:MP}, such a property
    entails that the {\em MP} holds.
\end{remark}




\section{Perturbation results}\label{sec:perturbation}

In this and the next section, exploiting the connections with the principal Floquet bundle provided by 
Theorems~\ref{thm:caracmu}, \ref{thm:caraceta}, \ref{thm:caraclambda},
we derive some other useful properties of the \gpe s.

We first state two results on the dependence of the generalized principal eigenvalues 
with respect to perturbations of the coefficients of the operator or of the domain $\O$.
These will be derived from the corresponding perturbation results for $u_P$ obtained in \cite{Huska} for operators in divergence form.

In order to deal with perturbations of the coefficients, we consider another operator
$$\tilde{P}=\partial_t-\tilde{a}_{ij}(t,x)\partial_{ij}-\tilde{b}_i(t,x)\partial_i-\tilde{c}(t,x).$$
We emphasize the dependence of the 
generalized principal eigenvalues on the operator $P$ or $\tilde P$
by including the operator in the notation.  
We will assume that $a_{ij}$ and $\tilde{a}_{ij}$ are Lipschitz-continuous in~$x$,
and that the operators fulfill the hypotheses of \cite{Huska} when written in 
divergence form.
%
We have the following.

\begin{corollary}\label{cor:continuity} 
%
     Let $I$ be $\R,\;\R^+$ or $\,\R^-$. 
     Assume that the coefficients of both operators $P$ and~$\t P$ satisfy (\ref{hyp:reg}), 
together with $\nabla a_{ij}, \nabla \tilde{a}_{ij} \in L^\infty(I\times \O)$.
For all $\e>0$, there exists $\delta>0$ such that if 
$$\|a_{ij}-\tilde{a}_{ij}\|_{ L^\infty(I\times \O)}<\delta,
\quad \|\nabla a_{ij}-\nabla \tilde{a}_{ij}\|_{L^{\infty}(I\times \O)}<\delta, \  \
\quad\|b_{i}-\tilde{b}_{i}\|_{L^{\infty}(I\times \O)}<\delta,
\quad \|c-\tilde{c}\|_{L^{\infty}(I\times \O)}<\delta, $$
then there holds 
 $$ |\mu_{b,p}(P,I)-\mu_{b,p}(\tilde{P},I)|<\e, \quad 
   |\lambda_{b,p}(P,I)-\lambda_{b,p}(\tilde{P},I)|<\e, $$ 
   $$ \text{if }I\neq\R^-,\quad
   |\mu_p(P,I)-\mu_{p}(\tilde{P},I)|<\e, \quad 
     |\lambda_b(P,I)-\lambda_{b}(\tilde{P},I)|<\e,$$
   $$ \text{if }I\neq\R^+,\quad
   |\mu_b(P,I)-\mu_b(\tilde{P},I)|<\e, \quad 
 |\lambda_p(P,I)-\lambda_{p}(\tilde{P},I)|<\e.$$ 
\end{corollary}

The above result is a consequence of  
Theorems \ref{thm:caracmu}, \ref{thm:caraceta}, \ref{thm:caraclambda} and 
of \cite[Proposition~8.3]{Huska}, which 
asserts that, putting the operators in divergence form, for any $\e>0$, 
there are some quantities $C,\delta>0$ such that if
$$\|a_{ij}-\tilde{a}_{ij}\|_{W^{1,\infty}(\R\times \O)}<\delta,
\quad\|b_{i}-\tilde{b}_{i}\|_{L^{\infty}(\R\times \O)}<\delta,
\quad \|c-\tilde{c}\|_{L^{\infty}(\R\times \O)}<\delta, $$
then  for all $s\in\R$, $t>0$, it holds
\Fi{uPutP}
\frac{\|u_{P} (s+t,\cdot)\|_{L^{\infty}(\O)}}{\| u_{P} (s,\cdot)\|_{L^\infty(\O)}}
\leq C e^{\e t}
\frac{\|u_{\t P} (s+t,\cdot)\|_{L^{\infty}(\O)}}{\| u_{\t P} (s,\cdot)\|_{L^\infty(\O)}}
\Ff
(in the case $I\neq\R$, this property is applied to the
even extension of the coefficients of $P$).

Next, we consider another bounded Lipschitz domain $\tilde{\O}$. 
We then highlight the dependence of the generalized principal eigenvalues 
on the domain by specifying in the notation to which one they correspond. 
%
The perturbation result is then derived from 
\cite[Proposition~9.5]{Huska}, which gives the estimate
\eqref{uPutP} provided that $\O$, $\t\O$ are close enough, in terms of the
Hausdorff distance $d(\partial \O, \partial \tilde{\O})$
of their boundaries.

\begin{corollary} \label{coro:regO} 
     Let $I$ be $\R,\;\R^+$ or $\,\R^-$. 
Assume (\ref{hyp:reg}) holds in $I\times(\O\cup\tilde\O)$
and $\nabla a_{ij} \in L^\infty(I\times (\O\cup\tilde\O))$. For all $\e>0$, there exists $\delta>0$ such that if 
$$d(\partial \O, \partial \tilde{\O})<\delta, $$
then there holds 
 $$ |\mu_{b,p}(P,I\times\t\O)-\mu_{b,p}(P,I\times\O)|<\e, \quad 
   |\lambda_{b,p}(P,I\times\t\O)-\lambda_{b,p}(P,I\times\O)|<\e, $$ 
   $$ \text{if }I\neq\R^-,\quad 
   |\mu_p(P,I\times\t\O)-\mu_{p}(P,I\times\O)|<\e, \quad 
     |\lambda_b(P,I\times\t\O)-\lambda_{b}(P,I\times\O)|<\e,$$
     $$  \text{if }I\neq\R^+,\quad
     |\mu_b(P,I\times\t\O)-\mu_b(P,I\times\O)|<\e, \quad 
 |\lambda_p(P,I\times\t\O)-\lambda_{p}(P,I\times\O)|<\e.$$ 
\end{corollary}



\section{Computation of the principal eigenvalues
for some classes of operators}\label{sec:cases}
The properties we have established for the \gpe s allow us to compute them for several important classes of operators. This is the purpose of this section.

\subsubsection*{Time-independent coefficients}

When the coefficients do not depend on $t$, we recover the classical notion of Dirichlet principal eigenvalue. 

\begin{proposition}\label{pro:elliptic}
Assume that the coefficients of $P$ do not depend on $t$, i.e., $P=\partial_t-\L$ with
$\L$ given by \eqref{L}. Then, there holds
$$\mu_{b,p}(\R)=\mu_{b,p}(\R^\pm)=
\lambda_{b,p}(\R)= \lambda_{b,p}(\R^\pm)=
\mu_{p}(\R^+)=\lambda_{b}(\R^+)=
\mu_{b}(\R^-)=\lambda_{p}(\R^-)= \lambda_{D}(-\L,\O),$$
where $\lambda_{D}(-\L,\O)$ is the Dirichlet principal eigenvalue of the elliptic operator $-\L$ in $\O$. 
\end{proposition}

This result is a particular case of Proposition \ref{prop:indet} below.


\subsubsection*{Time-independent principal part}

When the operator $P$ is the sum of an elliptic operator
plus a time-dependent zero order term $\sigma$, 
then the generalized principal eigenvalues 
can also be determined through various notions of time average of $\sigma$,
as the next result asserts.

\begin{proposition}\label{prop:indet}
Assume that $P=\partial_t-\L-\sigma(t)$, with
$\L$ given by \eqref{L} and $\sigma\in L^\infty(\R)$. 
Then, for $I= \R,\;\R^+$ or $\R^-$, there holds
$$\mu_{b,p}(I)=\lambda_D(-\L,\O) -\lim_{t\to +\infty}
\,\inf_{s,s+t\in I}\frac{\int_{s}^{s+t}\sigma(s')ds'}{t}\,,$$
$$\lambda_{b,p}(I)= \lambda_D(-\L,\O)-\lim_{t\to +\infty}
\,\sup_{s,s+t\in I}\frac{\int_{s}^{s+t}\sigma(s')ds'}{t}\,,$$
and
$$\mu_{p}(\R^+) = \lambda_D(-\L,\O)-
\liminf_{t\to +\infty}\frac{\int_{0}^{t}\sigma(s)ds}{t},\qquad 
\lambda_{b}(\R^+) = \lambda_D(-\L,\O)-\limsup_{t\to +\infty}\frac{\int_{0}^{t}\sigma(s)ds}{t}\,,$$
$$\mu_{b}(\R^-) = \lambda_D(-\L,\O)-
\liminf_{t\to -\infty}\frac{\int_{0}^{t}\sigma(s)ds}{t}, 
\qquad\lambda_{p}( \R^-) = \lambda_D(-\L,\O)-\limsup_{t\to -\infty}\frac{\int_{0}^{t}\sigma(s)ds}{t}\,,$$
where $\lambda_{D}(-\L,\O)$ is the Dirichlet principal eigenvalue of the elliptic operator $-\L$ in $\O$. 
\end{proposition}

\begin{proof}
It is straightforward to check that the function $u_P$ provided by Theorem \ref{thm:uP}
is given~by 
\Fi{uP-Dirichlet+sigma}
u_{P}(t,x)= \phi_{D}(x) e^{-\lambda_{D}(-\L,\O) t+\int_{0}^{t}\sigma(s)ds},
\Ff
where $\phi_{D}$ is the Dirichlet principal eigenfunction of $-\L$ in $\O$,
suitably normalized.
The equalities then follow from Theorems \ref{thm:caracmu}, \ref{thm:caraceta}, \ref{thm:caraclambda}. 
\end{proof}

\begin{remark}\label{rmk:differ}
    It follows from Proposition \ref{prop:indet} that all the notions 
    $\mu_{b,p}$, $\lambda_{b,p}$, $\mu_{p}$, $\lambda_b$, $\mu_{b}$, $\lambda_p$ 
    differ in general. It is indeed sufficient to choose $\sigma$ so that
    the corresponding notions of average are all different.
    As a consequence of this fact, 
    one infers that in general there does not exist a {\em 
    principal eigenfunction} associated with the 
    \gpe s $\mu_{b,p}$, $\lambda_{b,p}$, that is,
    a positive eigenfunction $\phi$ satisfying the requirements
    $\sup_{I\times\O}\phi<+\infty$ and $\inf_{I\times B}\phi>0$.
    Indeed, if it existed, then one would have $\mu_{b,p}\leq\lambda_{b,p}$, hence 
    by Proposition \ref{pro:basic} the eigenvalues $\mu_{b,p}$ and $\lambda_{b,p}$
    would coincide. 
    Likewise, in general there are no principal eigenfunctions
    associated with the other notions of \gpe s. For if they did exist, one would get $\mu_b(\R)\leq\lambda_b(\R)$
    and $\mu_p(\R)\leq \lambda_p(\R)$ but by Proposition~\ref{prop:indet}
    one can choose $\sigma$ in such a way that these inequalities fail.
\end{remark}



\subsubsection*{Coefficients converging when $t\to \pm \infty$}

For the next result, we need the Lipschitz regularity of the $a_{ij}$, that
allows one to write the operator in self-adjoint form.

\begin{proposition}\label{prop:infty}
Assume that $\nabla a_{ij} \in L^\infty(I\times \O)$, and that
$a_{ij}(t,x)\to a_{ij}^{\infty}(x)$, $\nabla a_{ij}(t,x)\to \nabla a_{ij}^{\infty}(x)$, 
$b_i(t,x)\to b_i^\infty(x)$, $c(t,x)\to c^\infty(x)$ 
as $t\to +\infty$ (resp.~$t\to-\infty$),
 uniformly in $x\in\O$.
Then
$$\mu_{b,p} (\R^+)=\lambda_{b,p} (\R^+)=
\mu_p (\R^+)=\lambda_b (\R^+) =  \lambda_{D}(-\L^{\infty},\O)
$$
$$
\text{(resp. }
\mu_{b,p} (\R^-)=\lambda_{b,p} (\R^-)=
\mu_b(\R^-)=\lambda_p(\R^-) =  \lambda_{D}(-\L^{\infty},\O)\text{ )},
$$
where $\lambda_{D}(-\L^{\infty},\O)$ is the Dirichlet principal eigenvalue of the elliptic operator 
$$-\L^{\infty}=-a_{ij}^{\infty}(x)\partial_{ij}-b_i^{\infty}(x)\partial_i-c^{\infty}(x)$$
in the domain $\O$. 
%
\end{proposition}

\begin{proof} 

As noted in 
Remark~\ref{rk:ball}$(\ref{rk:IT})$, one has $\mu_{b.p}(\R^+)=\mu_{b,p}((a,+\infty))$ for 
all $a>0$. 
Thus, applying Corollary \ref{cor:continuity} with $I=(a,+\infty)$, one deduces
$$\mu_{b,p}(\R^+)=\mu_{b,p}((a,+\infty))\xrightarrow{a\to+\infty}
\mu_{b,p}(\partial_t-\L^{\infty},\R^+),$$
which coincides with $\lambda_{D}(-\L^{\infty},\O)$ thanks to 
Proposition \ref{pro:elliptic}.
The other equivalences follow from the same arguments.
\end{proof}


\subsubsection*{Periodic coefficients}

We assume in this section that the coefficients are periodic in $t$,
with the same period $T>0$, namely
$$\forall t\in \R,\ x\in \O,\ i,j=1,...,N,\quad
a_{ij}(t+T,x)=a_{ij}(t,x),\quad
b_{i}(t+T,x)=b_i(t,x),
\quad c(t+T,x)=c(t,x).$$
Under this hypothesis, it turns out that
all the notions of principal eigenvalues are equivalent, coinciding with the
standard periodic principal eigenvalue.

\begin{proposition}\label{pro:periodic}
 Assume that the coefficients of $P$ are  periodic in $t$, with period $T$. 
 Then there holds
 $$\lambda_{b,p}(\R)= \lambda_{b}(\R)=\lambda_{p}(\R)=\mu_{b,p}(\R)=\mu_{p}(\R)=\mu_{b}(\R)= \lambda_{per}(\O),$$
 where $\lambda_{per}(\O)$ is the unique $\lambda\in \R$ such that there exists a solution $\phi$ 
 of the problem
 \[
 \left\{\begin{array}{ll}
 P\phi = \lambda\phi \ &\hbox{in } \R\times \O, \\
 \phi \;\text{ is periodic in $t$}  &\hbox{with period $T$},\\
 \phi (t,x)=0\  &\hbox{for } (t,x)\in \R \times \partial\O,\\
  \phi (t,x)>0 \  &\hbox{for } (t,x)\in \R \times \O.\\
 \end{array}\right.
 \]
 Moreover, the principal eigenvalue admits the following characterization:
 \[
\begin{split}
\lambda_{per}(\O) &=\sup\{\lambda\ :\ \ex\phi>0 \text{ in }\R\times\O,\ 
	\phi \hbox{ is $t$-periodic with period $T$}, \ P\phi\geq\lambda\phi\text{ in }\R\times\O\}
 \\[5pt]
 &=\inf\{\lambda\ :\ \ex\phi>0 \text{ in }\R\times\O,
 \ \phi=0\text{ on }\R\times\partial\O,\ \phi \hbox{ is $t$-periodic with period $T$},\\
 &\qquad  \hspace{284pt}
	P\phi\leq\lambda\phi\text{ in }\R\times\O\}.\hfill\\
\end{split}
\]
\end{proposition}

\begin{proof}  The existence and uniqueness
of the solution $\lambda_{per}(\O)$
of the above eigenproblem (together with its simplicity)
follows from the standard Krein-Rutman theory, since the problem is set on
a compact domain, due to periodicity. Let $\phi$ be the 
eigenfunction associated with $\lambda_{per}(\O)$. 
    It is immediately checked that $u_{P}(t,x)= \phi(t,x) e^{-\lambda_{per}(\O)t}$.
    The equalities 
    stated in the first part of the proposition are then a consequence of 
    Theorems \ref{thm:caracmu}, \ref{thm:caraceta}, \ref{thm:caraclambda}. 
The last equalities then follow too, because
    \[\begin{split}
\lambda_{per}(\O)&=\mu_{b,p}(\R)\\ &\leq 
\inf\{\lambda\ :\ \ex\phi>0 \text{ in }\R\times\O,
 \ \phi=0\text{ on }\R\times\partial\O,\ \phi \hbox{ is $t$-periodic with period $T$},\\
&\qquad  \hspace{284pt}
	P\phi\leq\lambda\phi\text{ in }\R\times\O\}\\
 &\leq \lambda_{per}(\O)\\
 &\leq\sup\{\lambda\ :\ \ex\phi>0 \text{ in }\R\times\O,\ 
	\phi \hbox{ is $t$-periodic with period $T$}, \ P\phi\geq\lambda\phi\text{ in }\R\times\O\}
 \\ &\leq \lambda_{b,p}(\R)=\lambda_{per}(\O).
\end{split}
\]
\end{proof}


\subsubsection*{Uniquely ergodic coefficients}

We now address the case of {\em uniquely ergodic} coefficients with respect to $t$.

\begin{definition} A function $f: I\times \O \to \R^{m}$ is called {\em uniquely ergodic} with respect to $t\in I$ if it is  uniformly continuous and bounded over $I\times \O$, and if there exists a unique invariant probability measure on its hull $\mathcal{H}_{f}:=cl\{ \tau_{a}f , a\in I\}$, where 
$cl$ is the closure with respect to the locally uniform convergence in $I\times\ol\O$, 
and where the invariance is understood with respect to the translations $\tau_{a} f (t,x):= f(t+a,x)$. 
\end{definition}

This notion is well-known to generalize that of {\em almost periodic} functions. Indeed, if $f: I\times \O \to \R^{m}$ is almost periodic with respect to $t$, let $\Psi:\mathcal{H}_{f}\to \R$. Then we could define $\int_{\mathcal{H}_{f}}\Psi d \mathbb{P}:= \lim_{T\to \infty}\frac{1}{2T}\int_{a-T}^{a+T} \Psi (\tau_t f)dt$ uniformly with respect to $a\in \R$. Moreover, one can check that $\mathbb{P}$ is the unique invariant probability measure on $\mathcal{H}_{f}:=cl\{ \tau_{a}f , a\in I\}$.

When the coefficients of $P$ are uniquely ergodic in $t$, we are able to prove that all the notions of generalized principal eigenvalues coincide.

\begin{proposition}\label{pro:ap}
Assume that the coefficients $(a_{ij})_{i,j\in \{1,\dots,N\}}$, $(b_{i})_{i\in \{1,\dots,N\}}$ and $c$ are uniquely ergodic with respect to $t\in I$ locally uniformly in $x\in \O$. Then there holds
$$\mu_{b,p} (\R)=\lambda_{b,p} (\R)=\mu_p (\R)=\lambda_p (\R) = \mu_b(\R)=\lambda_p(\R). $$
\end{proposition}

In the work \cite{HST}, the authors investigate the principal spectrum 
when the coefficients are periodic or almost periodic in time. 
They prove that the principal spectrum is a point, from which one could derive Proposition \ref{pro:ap} when the coefficients are almost periodic in time. Moreover, some comparisons are derived in \cite{HST} between the principal spectrum and the Dirichlet eigenvalues associated with coefficients that are averaged in time.

\begin{proof} 
This follows from \cite[Proposition 2.11 and Theorem 2.15]{Map}. Indeed, in that paper, the author defines a principal spectrum $E:=[\underline{\lambda}, \overline{\lambda}]$ and characterizes it in Proposition 2.11 as 
$$\underline{\lambda}=\lim_{t\to +\infty}
\left(\inf_{s\in \R}\frac{\ln \|v (s+t,\cdot)\|_{L^{2}(\O)} - 
	\ln \| v (s,\cdot)\|_{L^2(\O)}}{t}\right)$$
    $$\overline{\lambda}=\lim_{t\to +\infty}
\left(\sup_{s\in \R}\frac{\ln \|v (s+t,\cdot)\|_{L^{2}(\O)} - 
	\ln \| v (s,\cdot)\|_{L^2(\O)}}{t}\right)$$
where $v$ is the unique positive time-global solution of $Pv=0$, with $v=0$ on $\R\times \partial\O$ and $\|v(0,\cdot)\|_{L^2 (\O)}=1$ (constructed in  \cite[Theorem 2.4]{Map}). By uniqueness, there exists $\sigma>0$ such that $v\equiv \sigma u_P$, and we thus get, by Theorem \ref{thm:caracmu} and the subsequent remark, that $\mu_{b,p}(\R)=-\underline{\lambda}$ and $\lambda_{b,p}(\R)=-\overline{\lambda}$. Lastly, \cite[Theorem 2.15]{Map} yields $\underline{\lambda}=\overline{\lambda}$ and the conclusion then follows
using also Theorems \ref{thm:caraceta}, \ref{thm:caraclambda}. 
\end{proof}

\begin{remark} The identity $\mu_{b,p} (I)=\lambda_{b,p} (I)=:\lambda$ does not imply the existence of a positive function~$\phi$ satisfying $P\phi=\lambda \phi$ on $I\times\O$, $\phi=0$ on $I\times \partial \O$, and 
$\inf_{t\in I}\phi(t,x)>0$ for $x\in \O$, $\sup\phi<+\infty$.
Indeed, for instance, in the proof of 
    \cite[Proposition 3.5]{Rossiap}, the third author of the present paper exhibits the existence of an odd, almost periodic function $\sigma:\R\to\R$ such that $\frac1t\int_0^t \sigma (s)ds\to 0$ as $t\to +\infty$,
    but $\int_0^t \sigma (s)ds\to +\infty$ as $t\to +\infty$. 
    Then, if a function $\phi$ with the above properties existed for the operator $P=\partial_t-\Delta -\sigma (t)$
    with $I=\R$ or $I=\R^+$, then Theorem~\ref{thm:Huska} would imply that $\phi(t,x)e^{-\lambda t}=
    u_P(t,x)(q+w(t,x))$,
    for some (necessarily positive) constant $q$ and a function $w$ converging exponentially to $0$ as $t\to+\infty$.
    On the other hand, as seen in the proof of Proposition~\ref{prop:indet}, in this case the function $u_P$ is given by
    \eqref{uP-Dirichlet+sigma}, whence 
    $$\phi(t,x)\sim q\phi_{D}(x) e^{(\lambda-\lambda_{D}(-\L,\O))t+\int_0^t \sigma(s)ds}\quad\text{as }\;t\to+\infty,$$
    where $\lambda_{D}(-\L,\O)$, $\phi_{D}$ are the Dirichlet principal eigenvalue and eigenfunction of $-\L$ over $\O$.
    But then the properties of $\int_0^t\sigma$ prevent $\phi$ from simultaneously fulfilling the conditions 
    $\inf_{t\in \R^+}\phi(t,x)>0$ and $\sup\phi<+\infty$.
\end{remark}




\subsubsection*{Random stationary ergodic coefficients}

Consider the operator
$$P=\partial_t-a_{ij}(t,x,\o)\partial_{ij}-b_i(t,x,\o)\partial_i-c(t,x,\o).$$
We assume that the coefficients 
$(a_{ij})_{i,j\in \{1,\dots,N\}}$, $(b_{i})_{i\in \{1,\dots,N\}}$ and $c$ are random variables, 
defined for $(t,x,\omega)\in \R\times \O\times \mathcal{O}$,
where $(\mathcal{O},\mathbb{P},\mathcal{F})$ is a probability space.
We suppose that the hypotheses \eqref{hyp:reg} 
are satisfied almost surely, i.e.~for almost every $\o\in\O$ (with respect to the 
probability measure $\mathbb{P}$). 

The functions $(a_{ij})_{i,j\in \{1,\dots,N\}}$, $(b_{i})_{i\in \{1,\dots,N\}}$ and $c$ are assumed to be random stationary ergodic with respect to $t\in \R$. 
The stationarity hypothesis means that there exists a group $(\pi_{t})_{t\in \R}$ of measure-preserving
transformations  
such that  $a_{ij}(t+s,x,\omega)=a_{ij}(t,x,\pi_{s} \omega)$, $b_{i}(t+s,x,\omega)=b_{i}(t,x,\pi_{s} \omega)$ and $c(t+s,x,\omega)=c(t,x,\pi_{s} \omega)$ for all $(t,s,x,\omega)\in \R\times \R\times \O\times \mathcal{O}$. 
This hypothesis heuristically means that the statistical properties of the medium do
not depend on time at which one observes it. The ergodicity hypothesis means that if $\pi_{t} A=A$ for all $t\in \R$ and for a given  
$A\in\mathcal{F}$, then $\mathbb{P}(A)=0$ or $1$. 

For a given event $\o\in\mc{O}$, we let $P^\o$ denote the operator with coefficients
$a_{ij}(\cdot,\omega)$, $b_{i}(\cdot,\omega)$, $c(\cdot,\omega)$,
and stress the dependence on the operator in the notation of the \gpe s.

\begin{proposition}\label{pro:ergo}
Assume that the coefficients of $P$
are random stationary ergodic with respect to $t\in \R$. Then, for almost every event $\omega \in \mathcal{O}$, there holds
$$\mu_p (P^\o,\R)=\lambda_p (P^\o,\R) = \mu_b(P^\o,\R)=\lambda_p(P^\o,\R). $$
\end{proposition}

This case has been addressed in \cite{MS} in the framework of Floquet bundles, from which the identities stated in Proposition \ref{pro:ergo} can be derived owing to our results of Section \ref{sec:caracHuska}. We include the 
short proof here for the sake of completeness. 

\begin{proof} 
Let $u_P(\.,\o)$, $\o\in\mc{O}$, be the entire solution associated with the operator $P^\o$,
provided by Theorem~\ref{thm:uP}.
It easily follows from the uniqueness of $u_P$ that 
$$u_P (t+s, x , \omega)=u_P(t,x,\pi_s \omega)\|u_P (s,\cdot,\omega)\|_{L^\infty (\O)}.$$
Hence, 
$$\ln \|u_P (t+s,\cdot , \omega)\|_{L^\infty (\O)}=\ln\|u_P(t,\cdot,\pi_s \omega)\|_{L^\infty (\O)}+\ln\|u_P (s,\cdot,\omega)\|_{L^\infty (\O)}.$$
It then follows from the Birkhoff ergodic theorem that the limit
$$\lim_{t\to +\infty}\frac{1}{t}\ln \|u_P (t, \cdot , \omega)\|_{L^\infty (\O)}$$
exists almost surely and is deterministic.
The conclusion follows from Theorems \ref{thm:caraceta} and \ref{thm:caraclambda}.
\end{proof}

Note that there exist situations where 
$\lambda_{b,p}(P^\o,\R)>\lambda_p(P^\o,\R)>\mu_{b,p}(P^\o,\R)$ almost surely. 
As an example, consider a family of independent, identically distributed random variables 
$(\tilde{c}_k)_{k\in \mathbb{Z}}$ in $L^\infty (\tilde{\mathcal{O}})$ over a probability space $(\tilde{\mathcal{O}}, \tilde{\mathcal{F}},\tilde{\mathbb{P}})$. 
Let $\mathcal{O}:=\tilde{\mathcal{O}}^\mathbb{Z}\times \mathbb{T}$, with $\mathbb{T}:=\R/\mathbb{Z}$, and natural Borel space and probability over this space. 
For all $t\in \R$, consider the transformation $\pi_t:\O\to \O$ defined by 
$\pi_t( (\omega_k)_k,y):= ((\omega_{k+l})_k,y+r)$ if $t=l+r$, with $l\in \mathbb{Z}$ and $r\in [0,1)$. 
The ergodicity of $(\pi_t)_t$ is obvious. 
Consider the operator $P=\partial_t-\Delta-c(t,\omega)$, where 
$$c\big(t,((\omega_k)_k,y)\big):= (1-t-y+l)\t c_l (\omega_l)+
(t+y-l)\t c_{l+1}(\omega_{l+1})\quad \hbox{ for  }\;
t\in [l,l+1),\ l\in \mathbb{Z}.$$
One easily checks that $c(t+s,\omega)=c(t,\pi_{s} \omega)$ for all $(t,s,\omega)\in \R\times \R\times \mathcal{O}$. 

On one hand, we know from Proposition \ref{prop:indet} and the Birkhoff ergodic theorem that for almost every $\omega\in \mathcal{O}$, one has: 
$$\mu_{p}(P^\o,\R)=\mu_{b}(P^\o,\R) =\lambda_{p}(P^\o,\R) =\lambda_{b}(P^\o,\R)  = 
\lambda_D(-\Delta,\O)-\mathbb{E}[\t c],$$
where $\mathbb{E}[\t c]$ is the expectation of $\t c_k=\t c$ (which is independent of $k$ by identical distribution). 
But Proposition \ref{prop:indet} also yields
$$\mu_{b,p}(P^\o,\R)=\lambda_D(-\Delta,\O) -\lim_{t\to +\infty}\,\inf_{s\in \R}\frac{\int_{s}^{s+t}c(s', \omega)ds'}{t},$$
$$\lambda_{b,p}(P^\o,\R)= \lambda_D(-\Delta,\O)-\lim_{t\to +\infty}\,\sup_{s\in \R}\frac{\int_{s}^{s+t}c(s',\omega)ds'}{t}.$$
On the other hand, one can show that for all $\e>0$ and 
for almost every $\o=((\omega_k)_k,y)\in\mc{O}$, one can find 
arbitrary long consecutive sequences of $k's$ such that $\t c_k(\omega_k)\leq \essinf\t c+\e$ and arbitrary long consecutive sequences of $k's$ such that $\t c_k(\omega_k)\geq \esssup \t c-\e$. 
One infers that
$\mu_{b,p}(P^\o,\R)=\lambda_D(-\Delta,\O)-\essinf\t c$ and 
$\lambda_{b,p}(P^\o,\R)= \lambda_D(-\Delta,\O)-\esssup\t c$.



\section{The generalized principal eigenvalues of limit operators}\label{sec:*}
\subsection{Main result for limit operators}
We assume in this section that the second-order coefficients
$a_{ij}$ of $P$ are uniformly continuous over $I\times \O$.
We introduce the notion of {\em limit operator} associated with the operator $P$.
This is a parabolic operator 
$$P^*=\partial_t- a_{ij}^*(t,x)\partial_{ij}-b_i^*(t,x)\partial_i-c^*(t,x),$$
whose coefficients 
$b_i^*$, $c^*$ are the weak-$\star$ limits in $L^{\infty}$ as $n\to+\infty$ 
of $b_i(\.+t_n,\.)$, $c(\.+t_n,\.)$ respectively, and $a_{ij}^*$ is the strong limit in $\mathcal{C}^0_{loc}$ as $n\to+\infty$ 
of $a_{ij}(\.+t_n,\.)$,
for some sequence $\seq{t}$ in $I$ (which does not necessarily diverge). 
%
We let $\omega_{I}(P)$ denote the family of all limit operators associated with $P$.
Notice that, with respect to the usual definition of the omega-limit set of an operator, we also consider bounded sequence of translations $\seq{t}$, thus $\omega_I(P)$ contains in particular
the operator~$P$~itself.




In the next result, we stress the operator in the notation of the associated \gpe s.

\begin{theorem}\label{P*}
Assume that the coefficients of $P$ satisfy (\ref{hyp:reg}), 
and that the $ a_{ij}$ are uniformly continuous over $I\times \O$ for all $i,j$.
The following equivalences hold:

$$\mu_{b,p}(P,\R^+)=\max_{P^*\in\o_{\R^+}(P)}\mu_p(P^*,\R^+)
=\max_{P^*\in\o_{\R^+}(P)}\lambda_b(P^*,\R^+),$$
$$\mu_{b,p}(P,\R^-)=\max_{P^*\in\o_{\R^-}(P)}\mu_b(P^*,\R^-)
=\max_{P^*\in\o_{\R^-}(P)}\lambda_p(P^*,\R^-),$$
$$\lambda_{b,p}(P,\R^+)=\min_{P^*\in\o_{\R^+}(P)}\mu_p(P^*,\R^+)
=\min_{P^*\in\o_{\R^+}(P)}\lambda_b(P^*,\R^+),$$
$$\lambda_{b,p}(P,\R^-)=\min_{P^*\in\o_{\R^-}(P)}\mu_b(P^*,\R^-)
=\min_{P^*\in\o_{\R^-}(P)}\lambda_p(P^*,\R^-).$$

Moreover, if $P^*$ is a limit operator which realizes
one of the above maxima/minima, then it realizes the other maximum/minimum 
of the same chain of equivalences too.
\end{theorem}

The proof of Theorem \ref{P*} will be given in Section \ref{sec:proofopelim}.


We point out that the 
minimizers/maximizers $P^*$ appearing in the above expressions 
are not necessarily unique. Indeed, consider for example the case where $P=\partial_t-\Delta -c(t)$ with $c$
almost periodic in $t$. Then 
$$\mu_{p}(P^*,\R) = \lambda_D(-\Delta,\O)-\liminf_{t\to +\infty}\frac{\int_{0}^{t}c^*(s)ds}{t}$$ according to Proposition \ref{prop:indet} (using the same notations), where $c^*(t)=\lim_{n\to +\infty} c(t+t_n)$ for some sequence $(t_n)_{n\in\N}$
defining the limit operator $P^*$. But the almost periodicity of $c$ yields that this limit does not depend on the limit operator $P^*$. 
This shows that all the limit operators maximize~$\mu_{p}(P^*,\R)$.



\subsection{Relation with the exponential type}

Let us conclude by some comparison with the notion of exponential type. 
Rodr\'iguez-Bernal
and his coauthors studied exponentially stable operators and their links with semilinear parabolic equations under similar KPP-type hypotheses on $f$. 
We will not describe their results in details here but just describe some consequences in the case we address in the present paper.

In \cite{RB}, the notion of 
{\em exponential type} is discussed. Such a notion actually coincides 
with the principal eigenvalue $\lambda_{b,p}$ (up to a sign)
introduced in the present paper, as we now show. 
Note that in \cite{RB}, Neumann and Robin boundary conditions are also addressed.
Let us recall the definition of exponential type.
Let $P$ be a linear parabolic operator whose coefficients satisfy our standing assumptions \eqref{hyp:reg}.
For all $s\in\R$, we let $u^s=u^s(t,x;u_{0})$ denote 
the solution of $Pu^s=0$ in $(s,+\infty)\times \O$ with initial datum $u^s(s,x;u_{0})=u_{0}(x)$ for 
$x\in \O$. 
The exponential type $\beta_{P}$ (resp.~exponential type at $-\infty$ $\beta_{P}^{-}$ or at $+\infty$ $\beta_{P}^{+}$) is the smallest $\beta$ such that, for any compactly supported and continuous initial datum $u_{0}$, one has
\begin{equation}\label{eq:beta}\limsup_{t\to +\infty}
\left(\inf_{s\in \R}
\frac{\ln \|u^s(t+s,\.;u_{0})\|_{L^{1}(\O)}-\ln \|u_{0}\|_{L^{1}(\O)}}{t}\right)\leq \beta\end{equation}
(resp.~with $s\in\R^-$ or $s\in\R^+$ in the infimum). 
One has 
\begin{equation} \label{eq:exptype}
\beta_{P} = -\lambda_{b,p}(\R), \quad \beta_{P}^{-} = - \lambda_{b,p}(\R^{-}), \quad 
 \beta_{P}^{+} = - \lambda_{b,p}(\R^{+}).\end{equation}

Indeed, applying Theorem \ref{thm:Huska} with $u_{P}(t,\cdot)$ replaced by
$u_{P}(t+s,\cdot)/\|u_{P}(s,\cdot)\|_{L^{\infty}(\O)}$, one gets a constant $q_s\in\R$ such that
$$\forall t>0,\quad
\|u^s(t+s,\.;u_0)-q_s u_{P}(t+s,\.)\|_{L^{\infty}(\O)}\leq 
C\|u_{0}-q_s u_{P}(s,\.)\|_{L^{\infty}(\O)}
\frac{\|u_{P}(t+s,\cdot)\|_{L^{\infty}(\O)}}{\|u_{P}(s,\cdot)\|_{L^{\infty}(\O)}}\,e^{-\gamma t}\,$$
where the constants $C,\gamma$ are independent of $s$, see 
\cite[Theorem~2.6(iii)]{Huska}. Moreover, $q_s$ is such that $u^s(t+s,\.;u_0)-q_s u_{P}(t+s,\.)$ changes sign in $\O$.
If $u_0$ is positive in $\O$, then
$q_s \|u_{P}(s,\.)\|_{L^{\infty}(\O)}$ can neither be too small nor too large, because
otherwise $u^s(t+s,\.;u_0)-q_s u_{P}(t+s,\.)$ would not change sign on $\O$
(as a consequence of the Harnack inequality, cf.~Theorem~\ref{thm:Harnack}).

It follows that the left-hand side in (\ref{eq:beta}) is maximized when $u_0=u_P(s,\cdot)$ and $u^s\equiv u_P$. According to Theorems \ref{thm:Huska} and \ref{thm:caracmu}, we thus get that $\beta_{P} = -\lambda_{b,p}(\R)$. The proofs for $ \beta_{P}^{\pm}$ are similar.

One could thus derive estimates on the generalized principal eigenvalue $\lambda_{b,p}$ 
for a particular class of operators
by using the following result of \cite{RB}.

\begin{proposition}[{\cite[Lemma 2.7]{RB}}]
Consider the operator
$P=\partial_t-\Delta-c(t,x)$, with $c\in L^\infty(\R\times\O)$.
For given $t\in \R$, let $\lambda_{D}(-\L_t)$ be the Dirichlet principal 
eigenvalue of the elliptic operator 
$$-\L_t=-\Delta-c(t,x)\quad\text{ in }\;\O.$$
Then one has 
$$-\beta_{P}  \geq \lim_{t\to +\infty}
\Big(\inf_{s\in \R}\frac{1}{t}\int_{s}^{s+t}\lambda_{D}(-\L_{t'})dt'\Big),$$
and
$$-\beta_{P}^\pm \geq \lim_{t\to +\infty}\inf_{s\in\R^\pm}
\frac{1}{t}\int_{s}^{s+t}
\lambda_{D}(-\L_{t'})dt'.$$
 \end{proposition}
 
 In~\cite{R-RB-VL}, Robinson, Rodr\'iguez-Bernal and Vidal-L\'opez
studied the following problem with logistic nonlinearity:
\Fi{logistic}
\begin{cases}
\partial_t u-\Delta u= c(t,x)u-n(t,x)u^3, & t\in \R,\ x\in \O,\\
u(t,x)=0,  & t\in \R, \ x\in\partial \O,
\end{cases}
\Ff
with $c,n\in L^\infty(\R\times\O)$ and H\"older continuous in $t$, and $n\geq0$. 
Calling $P:=\partial_t-\Delta u-c(t,x)$ the linearized operator,
\cite[Theorem 8.1]{R-RB-VL} implies that if $\beta_P^-<0$
then \eqref{logistic} 
does not admit any positive bounded ancient solution.
This result is indeed obtained by considering a nonnegative ancient solution
$u$ and applying \cite[Theorem 8.1]{R-RB-VL} to 
$u^{s_n}(t+s_n,\.;u(s_n,\.))$ on a suitable sequence of times $s_n\to-\infty$
(for~which what is called ``exponential stability'' in \cite{R-RB-VL}
holds).
We recover this result, and actually extend it in several directions, using ours.
Namely, for the more general problem \eqref{RD}
under the hypotheses \eqref{f=0}-\eqref{hyp:fKPP} on $f$ (which include \eqref{logistic}),
a sufficient condition for the non-existence of 
positive bounded ancient solutions is $\mu_b(\R^-)>0$.
This is an immediate consequence of Theorem \ref{thm:MP}, since
solutions to \eqref{f=0} are subsolutions for the linearized operator $P$ 
thanks to the KPP condition~\eqref{hyp:fKPP}.
This improves the result of \cite{R-RB-VL}
because, by Proposition \ref{pro:basic}, $\mu_b(\R^-)\geq\lambda_{b,p}(\R^-)=-\beta_P^-$.
%


\section{Technical tools}\label{sec:tool}


\subsection{H\"older-continuity of the Floquet bundles}

Let $u_P$ be the time-global, positive solution provided by \thm{uP}.
Recall that if $P$ is just defined on $\R^+\times\O$ or on $\R^-\times\O$,
then it is extended by even reflection with respect to $t$.
The following function incorporates some crucial information about the dynamical properties of the equation:
\Fi{beta}
\beta(t):= \ln(\|u_{P} (t,\cdot)\|_{L^{\infty}(\O)}).
\Ff

We will apply to the function $\beta$ some notions of average growth-rate, 
that require the function to have at most linear growth at infinity, cf.~condition~\eqref{hyp:sublinear}
below.
The latter property is ensured by the uniform continuity,
which is granted by the following result.

\begin{lemma}\label{lem:beta-Holder}
	The function $\beta$ is locally H\"older-continuous in $\R$ with some exponent
 $\alpha>0$, and it satisfies
 $$\sup_{s\in\R,\ t\in(0,1)}\frac{|\beta(s+t)-\beta(s)|}{t^\alpha}<+\infty.$$
\end{lemma}

\begin{proof}
    We start with applying parabolic estimates to the function $u_P$, which, we recall,
    is defined for all $t\in\R$. Namely, there is 
    a constant $C>0$ such that
    $$\forall s\in\R,\quad
    \|u_P\|_{W^{1,2}_{p}((s-1,s+1)\times\O)}\leq C \|u_P\|_{L^{p}((s-2,s+1)\times\O)},$$
 see e.g.\ \cite[Theorem 7.30]{Lie}.
 Then, since $p>N+1$, we get from Morrey's inequality
$$\forall s\in\R,\quad
\|u_P\|_{\mathcal{C}^{0,\alpha}((s-1,s+1)\times\O)}\leq C\|u_P\|_{L^{p}((s-2,s+1)\times\O)},$$
for some $\alpha>0$ and another constant independent of $s$, that we still call $C$.
It follows that
\Fi{uPalpha}
\forall s\in\R,\quad
\|u_P\|_{\mathcal{C}^{0,\alpha}((s-1,s+1)\times\O)}\leq C\|u_P\|_{L^{\infty}((s-2,s+1)\times\O)},
\Ff
for some other $C>0$, depending also on $|\O|$ and $p$.
We make now use of the following two-sided estimate quoted from
\cite[Corollary~3.10]{Huska}: there exists $C'>1$ such that
\Fi{C'<C'}
\forall s\in\R,\ t \in [0,1], \quad \frac{1}{C'} \|u_{P} (s,\cdot)\|_{L^{\infty}(\O)} \leq \|u_{P} (s+t,\cdot)\|_{L^{\infty}(\O)}
\leq C' \|u_{P} (s,\cdot)\|_{L^{\infty}(\O)}.
\Ff
We point out that the second inequality above is a straightforward consequence
of the comparison principle (even with $C'\to1$ as $t\to0^+$), whereas
the first one makes use of a boundary Harnack-type inequality.
Gathering together \eqref{uPalpha}, \eqref{C'<C'} one gets
$$\forall s\in\R,\quad
\|u_P\|_{\mathcal{C}^{0,\alpha}((s-1,s+1)\times\O)}\leq \hat C\|u_P(s,\.)\|_{L^{\infty}(\O)},$$
with $\hat C>0$ independent of $s\in\R$.
Then, since by the definition of $\beta$ it holds that
$$\beta(s+t)-\beta(s)
=\ln\Big(1+\frac{\|u_P(s+t,\.)\|_{L^{\infty}(\O)}-\|u_P(s,\.)\|_{L^{\infty}(\O)}}
{\|u_P(s,\.)\|_{L^{\infty}(\O)}}\Big),$$
one eventually derives, for $s\in\R$ and $|t|^\alpha<1/\hat C$,
$$\ln(1-\hat C |t|^\alpha)\leq
|\beta(s+t)-\beta(s)|
\leq \ln(1+\hat C |t|^\alpha).$$
This immediately yields the desired estimate.
%
\end{proof}

\color{black}

\subsection{Global growth-rate}\label{sec:gr}

The aim of this section is to derive an equivalent of \cite[Lemma 3.2]{NR1}
concerning the notion of the {\em least~mean}, 
introduced by the last two authors of the present paper. 
Let us first remind to the reader what that result is. 
If $g\in L^\infty (\R)$, then the least mean of $g$ is defined by
$$\lm{g}:=\lim_{t\to+\infty}\bigg(\inf_{s\in \R}\,\frac1t\int_s^{s+t} g(\tau)d\tau\bigg).$$
In \cite{NR1}, after showing that the above limit always exists, 
the following  characterization is derived:
\begin{equation} \label{eq:characlm}
\lm{g}=
\sup_{B\in W^{1,\infty}(\R)}\left(\essinf_{\R}(g+B')\right).\end{equation}
This plays a crucial role in the proofs of \cite{NR1}.
In the current paper, we would like to apply
the notion of the least mean to the derivative of the function 
$\beta(t):= \ln(\|u_{P} (t,\cdot)\|_{L^{\infty}(\O)})$,
which is not possible because it is not in $L^\infty$ in general.
For this reason, we will need to rewrite the notion of the least mean
of a function $g$ in terms of its primitive. This leads us to introduce
the notion of the {\em least global growth-rate}, together with the {\em greatest global growth-rate}, that we will apply to the function $\beta$.

\begin{definition}\label{def:lgr}
    Let $I$ be an unbounded open interval and
   let $G:I\to\R$ be a measurable function such that
    \begin{equation}\label{hyp:sublinear}
    \exists C>0,\quad
    \sup_{t_1,t_2\in I} 
    |G(t_1)-G(t_2)|\leq C(1+|t_1-t_2|).
    \end{equation}
    Then the following quantities:
    $$\lgr{G}_I:=\limsup_{t\to+\infty}\bigg(\inf_{s,s+t\in I}
    \,\frac{G(s+t)-G(s)}t\bigg),
    \qquad 
    \ggr{G}_I:=\liminf_{t\to+\infty}\bigg(\sup_{s,s+t\in I}
    \,\frac{G(s+t)-G(s)}t\bigg)$$
    (which exist and are finite) are called the {\em least global growth-rate} 
    and the {\em greatest global growth-rate} of $G$ over $I$ respectively.
    
\end{definition}

Notice that the least growth-rate of a function $G$ coincides with 
the least mean of its derivative $G'$, which is well defined if
$G$ is Lipschitz-continuous. Then, in such a case, 
one could apply the characterization \eqref{eq:characlm}
of \cite{NR1} to $G'$ (see also \cite[Lemma 2.2]{SalakoShen} for some generalizations 
to functions which are not necessarily bounded) and get an analogous characterization for the least global growth-rate of $G$. 
However, in the present paper, the function $G=\beta$ is not Lipschitz-continuous but only H\"older-continuous, and thus we cannot apply \eqref{eq:characlm}.
This is why we need to reformulate and extend the characterization of the
least global growth-rate of a function without passing through its
derivative. We will show by the way that the
``$\,\limsup$'' and ``$\,\liminf$''
in Definition \ref{def:lgr} are actually limits.

\begin{proposition}\label{pro:gr-char}
Let $G:\R\to\R$ be a measurable function
satisfying \eqref{hyp:sublinear}. There holds that
$$\lgr{G}_I
=\sup\left\{\essinf_{I}A'\ :\ A-G\in L^{\infty}(I),\ A'\in L^{\infty}(I)\right\}$$
and 
$$\ggr{G}_I
=\inf\left\{\esssup_{I}A'\ :\ A-G\in L^{\infty}(I),\ A'\in L^{\infty}(I)\right\}.$$
Moreover the ``$\,\limsup$'' and ``$\,\liminf$''
in the definitions of $\lgr{G}_I$ and $\ggr{G}_I$
are actually ``$\,\lim$''.

Finally, one has
	$$\lgr{G}_\R=\min\big\{\lgr{G}_{\R^-},\lgr{G}_{\R^+}\big\},
	\qquad
	\ggr{G}_\R=\max\big\{\ggr{G}_{\R^-},\ggr{G}_{\R^+}\big\}.$$
\end{proposition}

\begin{proof}
We prove the statements about the least global growth-rate.
Applying them to the function $-G$ one gets the results
for the greatest global growth-rate.

We first assume that $I\neq\R$. Since all the quantities involved in the statement are invariant by reflection of the function $G$, we can assume that $I=\R^+$. 

Take $m<\lgr{G}_{\R^+}$. By the definition of the least global growth-rate,
there exists $T>0$ such that 
$$m<\inf_{s\geq 0} \frac{1}{T}(G(s+T)-G(s)).$$

We now consider a linear interpolation of the function $G$.
Namely, for $t\geq0$, we define
$$
\forall k\in\N, \ t\in [(k-1)T,kT),\qquad
\alpha(t):= \frac{G(kT)-G((k-1)T)}T,$$
then we set
$$A(t):=G(0)+\int_0^t \alpha(s)ds.$$
%
%
%
The function $A$ coincides with $G$ over the set $T\N$.
Moreover, outside such a set, it holds $A'=\alpha$,
hence by \eqref{hyp:sublinear} there exists $C>0$
such that $|A'|=|\alpha|\leq C(1+1/T)$.
As a consequence, using again~\eqref{hyp:sublinear}, one finds,
for any $k\in\N$ and $t\in [(k-1)T,kT)$,
$$|A(t)-G(t)|\leq |A(t)-A((k-1)T)|+| G(t)-G((k-1)T)|\leq 2C(T+1),$$
which shows that $A-G$ is bounded over $\R^+$. 
We further know that $A'=\alpha>m$ outside $T\N$. 
We have thus proved that the inequality
$$ \sup\left\{\essinf_{\R^+}A'\ :\ A-G\in L^{\infty}(\R^+),\ A'\in L^{\infty}(\R^+)\right\}\geq m$$
holds for any $m<\lgr{G}_{\R^+}$, hence it holds true for $m=\lgr{G}_{\R^+}$.

In order to show the reverse inequality, 
consider an arbitrary unbounded interval $I$ and let $A$ be 
such that $A-G\in L^\infty (I)$ and $A'\in L^\infty (I)$. Then, for all $s\in I$ and $t> 0$ such that $s+t\in I$, one has
\[\begin{split}
G(s+t)-G(s) &\geq 
G(s+t)-A(s+t)-G(s)+A(s)+\int_s^{s+t}A'\\
&\geq -2\|A-G\|_{L^\infty(I)}+ t\,\essinf_{I}A',
\end{split}\]
whence
\Fi{lgr>}
\lgr{G}_I\geq\liminf_{t\to+\infty}\bigg(\inf_{s,s+t\in I}
    \,\frac{G(s+t)-G(s)}t\bigg)\geq \essinf_{I}A'.
    \Ff
Applying this lower bound with $I=\R^+$ one deduces at once 
the characterization for $\lgr{G}_{\R^+}$ ad also that
the ``$\,\limsup$''in the definitions of $\lgr{G}_I$ is actually a ``$\,\lim$''.

Let us now show the equivalence 
\Fi{lgr+-}
\lgr{G}_\R=\min\big\{\lgr{G}_{\R^-},\lgr{G}_{\R^+}\big\}.
\Ff
The inequality ``$\leq$'' in \eqref{lgr+-}
is a direct consequence of the definition of the involved quantities.
For the reverse one, consider any pair of functions 
$A_+$, $A_-$ satisfying $A_\pm-G\in L^{\infty}(\R^\pm)$ and $A_\pm'\in L^{\infty}(\R^\pm)$.
We then define
	\begin{equation} \label{eq:A}A(t):=\begin{cases}
	A_-(t) & \text{if }t\leq0\\
	A_+(t)-A_+(0)+A_-(0) & \text{if }t>0.
	\end{cases}\end{equation}
 This is a Lipschitz continuous function satisfying 
 $A-G\in L^{\infty}(\R)$, hence using \eqref{lgr>} with $I=\R$ gives
$$\lgr{G}_\R\geq \essinf_{\R}A'=\min\big\{\essinf_{\R^-}A_-'\,,\,\essinf_{\R^+}A_+'\big\}.$$
Taking the supremum with respect to the functions $A_-$, $A_+$ yields 
$\lgr{G}_\R\geq\min\{\lgr{G}_{\R^-},\lgr{G}_{\R^-}\}$. This proves~\eqref{lgr+-}.

It remains to prove the characterization for $\lgr{G}_\R$.
We already have one inequality, cf.~\eqref{lgr>}. Let us show  the opposite one. 
By the characterization for $\lgr{G}_{\R^\pm}$, for any $\e>0$
there exist two functions 
$A_\pm$ satisfying $A_\pm-G\in L^{\infty}(\R^\pm)$ and $A_\pm'\in L^{\infty}(\R^\pm)$, such that:
$$\lgr{G}_{\R^\pm}\leq \e+\essinf_{\R^\pm}A'_\pm.$$
Defining $A$ as in \eqref{eq:A}, one then gets from~\eqref{lgr+-} 
$$\lgr{G}_\R=\min\{\lgr{G}_{\R^-},\lgr{G}_{\R^-}\}\leq \e+\essinf_{\R}A',$$
whence
$$\lgr{G}_\R\leq \sup\left\{\essinf_{\R}A'\ :\ A-G\in L^{\infty}(\R),\ A'\in L^{\infty}(\R)\right\}
+\e.$$
Since $\e>0$ was arbitrary, we have obtained the desired upper bound for
$\lgr{G}_\R$, which concludes the proof.
\end{proof}


\section{Connections with Floquet bundles}
\label{sec:proofs}


All of our results concerning the relation between generalized principal eigenvalues and Floquet bundles
can be reformulated in terms of suitable averages of the growth 
rate of the function~$\beta$ defined by~\eqref{beta}.
In particular, Theorem~\ref{thm:caracmu} involves the notion of 
least and greatest global growth-rates introduced in Section~\ref{sec:gr}.
Another essential ingredient is the following Harnack-type inequality for quotients
of positive solutions, quoted from~\cite{Huska}.

\begin{theorem}[{\cite[Theorem 2.1]{Huska}}]\label{thm:Harnack}
Let $u_1,u_2$ be two positive solutions of \eqref{Pu=0} for $t>0$.
Then, for any $s_0>0$, it holds 
$$\forall s\geq s_0,\quad
\sup_{x\in \O}\frac{u_2(s,x)}{u_1(s,x)}\leq
C\inf_{x\in \O}\frac{u_2(s,x)}{u_1(s,x)},$$
where $C>0$ only depends on $s_0$, $N$, $\O$, the ellipticity constant $\alpha$ of $(a_{ij})$
and the $L^\infty$ bounds of the coefficients of the operator.
\end{theorem}

We point out that, unlike the standard parabolic Harnack inequality, 
the $\sup$ and $\inf$ in the above estimate are taken {\em at the same time }$s$.

We are now in a position to prove our main results.

\begin{proof}[Proof of Theorem \ref{thm:caracmu}]
        Lemma \ref{lem:beta-Holder}
        implies that the function $\beta$ is uniformly continuous, 
        hence it fulfills~\eqref{hyp:sublinear}. 
        We can then consider its least and greatest global growth-rates 
        given by Definition \ref{def:lgr}.
        The desired 
	equivalences then rewrites as
	$$\mu_{b,p}(I)=-\lgr{\beta}_I,\qquad
          \lambda_{b,p}(I)=-\ggr{\beta}_I.$$
\quad {\em The inequalities $\mu_{b,p}(I)\leq-\lgr{\beta}_I$ and 
$\lambda_{b,p}(I)\geq-\ggr{\beta}_I$.}\\
 Consider an arbitrary $\lambda>-\lgr{\beta}_I$.
	Te characterization of $\lgr{\beta}_I$ given by 
 Proposition \ref{pro:gr-char} provides us with a function $A$
 	satisfying 
        $A'\in L^{\infty}(I)$ and $A-\beta\in L^{\infty}(I)$,
        such that
 $A'>-\lambda$ a.e.~in $I$. Define the function
	$$\phi(t,x):=u_P(t,x)e^{-A(t)}.$$
	This function satisfies 
	$$P\phi=-A'\phi<\lambda\phi\quad\text{ in }I\times\O.$$
	Observe that
        $$\|\phi(t,\.)\|_{L^\infty(\O)}=e^{\beta(t)-A(t)},$$ 
        which is bounded for $t\in I$.
        We claim that $\phi$ is also bounded from below away from $0$
        on any fixed open ball $B\Subset \O$.
        The boundary Harnack inequality\footnote{ 
        \ Deduced from the localized one of \cite{bHarnack} (see also \cite[Theorem 3.5]{Huska})
        and extended to the whole $\O$ using a covering argument
        and the standard interior Harnack inequality.}
        yields the existence of a constant $\hat C>0$ such~that
$$
\forall  t\in I, \quad \|u_{P} (t-1,\cdot)\|_{L^{\infty}(\O)} \leq 
\hat C\inf_{x\in B} u_{P} (t,x),
$$
	and we know from (\ref{C'<C'}) that the left hand-side is bounded from below by 
	$\frac{1}{C'}  \|u_{P} (t,\cdot)\|_{L^{\infty}(\O)}$, for 
 some other constant $C'>0$.
 Hence, 
there exists a positive constant $C$ such that
	$$\forall  t\in I, \quad 
	\inf_{x\in B}\phi(t,x)= \frac{\inf_{x\in B} u_{P} (t,x)}{\|u_{P} (t,\cdot)\|_{L^{\infty}(\O)}}e^{\beta(t)-A(t)}
 	\geq C e^{\beta(t)-A(t)}.$$
	We have shown that the bounded function $\phi$ is also
bounded from below away from~$0$ on $B$, hence it 
 can be used in the definition of $\mu_{b,p}(I)$.
	We deduce that $\mu_{b,p}(I)\leq\lambda$, and this
	being true for any $\lambda>-\lgr{\beta}_I$ shows
	$\mu_{b,p}(I)\leq-\lgr{\beta}_I$.

 Proceeding exactly in the same way, but starting from 
 $\lambda<-\ggr{\beta}_I$ and then applying Proposition~\ref{pro:gr-char} to get 
 a function $A$ satisfying $A'<-\lambda$, one derives the
 inequality $\lambda_{b,p}(I)\geq-\ggr{\beta}_I$.

\medskip

\quad {\em The inequalities $\mu_{b,p}(I)\geq-\lgr{\beta}_I$ and 
$\lambda_{b,p}(I)\leq-\ggr{\beta}_I$.}\\
 Take $\lambda>\mu_{b,p}(I)$. Then 
        there exists $\phi$ such~that
        $$\phi>0\quad\text{and}\quad
	P\phi\leq\lambda\phi\quad\text{in }I\times\O,$$
	and moreover
 $$\phi=0\text{ on }I\times\partial\O,\qquad
	\inf_{t\in I,\ x\in B}\phi(t,x)\geq \frac1K \quad\text{and}\quad
 \sup_{t\in I,\ x\in \O}\phi(t,x)\leq K,$$ 
 for some open ball $B\Subset \O$ and some $K>0$.
Fix $s\in I$. 
	Consider the solution $u$ of the problem $Pu=0$ on $(s,+\infty)\times\O$
	under Dirichlet boundary condition and with initial datum 
	$u(s,\.)=\phi(s,\.)$ (as usual, in the case $I=\R^-$, 
 $P$ is extended to all times
 by even reflection).
	Since $\phi(s+t,x)e^{-\lambda t}$ and $Ke^{\norma{c}t}$
 are respectively a subsolution and a supersolution to such a problem,
 it follows from the comparison principle that 
	\Fi{<psi<}
 \forall t\geq0,\quad
	\frac1K e^{-\lambda t}\leq\|u(s+t,\.)\|_{L^{\infty}(\O)}\leq
 Ke^{\norma{c}t}.
	\Ff
        Next, applying the Harnack-type inequality provided by
        Theorem~\ref{thm:Harnack}, we find 
        a constant $C>0$, independent of $s$, such that 
$$m_s:=\sup_{x\in \O}\frac{u(s+1,x)}{u_P(s+1,x)}\leq
C\inf_{x\in \O}\frac{u(s+1,x)}{u_P(s+1,x)}\leq
\frac{C Ke^{\norma{c}}}{\|u_P(s+1,\.)\|_{L^{\infty}(\O)}}.$$
Thus, always by comparison, $m_s u_P(s+t,x)\geq u(s+t,x)$
for $t\geq1$, $x\in\O$, whence using \eqref{<psi<}
$$\forall t\geq1,\quad
\|u_P(s+t,\.)\|_{L^{\infty}(\O)}\geq \frac{1}{m_s}\|u(s+t,\.)\|_{L^{\infty}(\O)}\geq
\frac{e^{-\lambda t}}{CK^2e^{\norma{c}}}
\,\|u_P(s+1,\.)\|_{L^{\infty}(\O)}.$$
This inequality rewrites in terms of $\beta$ as
$$\beta(s+t)-\beta(s+1)=\ln\frac{\|u_P(s+t,\.)\|_{L^{\infty}(\O)}}
{\|u_P(s+1,\.)\|_{L^{\infty}(\O)}} \geq 
-\lambda t-\ln(CK^2e^{\norma{c}}),$$
where, we recall, $C,K$ are independent of $s$. 
Owing to Lemma \ref{lem:beta-Holder}, there exists $C'$ independent of
$s$ such that $|\beta(s+1)-\beta(s)|\leq C'$.
As a consequence, we have
$$\inf_{s\in I}\frac{\beta(s+t)-\beta(s)}t\geq 
-\lambda-\frac1t \ln(CK^2e^{\norma{c}})-\frac{C'}t\to-\lambda\qquad\text{as }\;t\to+\infty.$$
One deduces $\lgr{\beta}_I\geq-\lambda$, which, 
being true for any $\lambda>\mu_{b,p}(I)$, implies that
$\mu_{b,p}(I)\geq-\lgr{\beta}_I$.

The proof of $\lambda_{b,p}(I)\leq-\ggr{\beta}_I$
is similar. 
 Take $\lambda<\lambda_{b,p}(I)$. Then 
        there exists $\phi$ such~that
	$$\phi>0\quad\text{and}\quad
	P\phi\geq\lambda\phi\quad\text{in }I\times\O,$$
	and moreover
 $$	\inf_{t\in I,\ x\in B}\phi(t,x)\geq \frac 1K \quad\text{and}\quad
 \sup_{t\in I,\ x\in \O}\phi(t,x)\leq K,$$ 
 for some open ball $B\Subset \O$ and some $K>0$.
The solution to $Pu=0$ on $(s,+\infty)\times\O$
	under Dirichlet boundary condition and with initial datum 
	$u(s,\.)=\phi(s,\.)$ satisfies
	\Fi{>psi>}
 \forall t\geq0,\ x\in\O,\quad
	u(s+t,x)\leq \phi(s+t,x)e^{-\lambda t}\leq
 K e^{-\lambda t}.
	\Ff
 Moreover, since $\phi(s,x)\geq 1/K$ for $x\in B$, it is easily seen (for instance arguing by contradiction)
 that there exists another constant $K'>0$
 independent of $s$ such that $u(s+1,x)\geq K'$ for $x\in B$.
        Thus, by
        Theorem~\ref{thm:Harnack}, there exists
        a constant $C>0$, independent of $s$, such that 
$$M_s:=\sup_{x\in \O}\frac{u_P(s+1,x)}{u(s+1,x)}\leq
C\inf_{x\in \O}\frac{u_P(s+1,x)}{u(s+1,x)}\leq
\frac{C}{K'}\|u_P(s+1,\.)\|_{L^{\infty}(\O)}.$$
By comparison and using \eqref{>psi>} we get
$$\forall t\geq1,\ x\in\O,\quad 
u_P(s+t,x)\leq M_su(s+t,x)\leq M_s K e^{-\lambda t}
\leq \frac{C}{K'K}\|u_P(s+1,\.)\|_{L^{\infty}(\O)}e^{-\lambda t},$$
whence
$$\forall t\geq1,\ \quad 
\beta(s+t)-\beta(s+1)\leq -\lambda t-\ln(CK/K').$$
This entails that $\ggr{\beta}_I\leq-\lambda$, and finally $\lambda_{b,p}(I)\leq-\ggr{\beta}_I$.
\end{proof}

\begin{remark}\label{rk:relax_bp}
As a matter of fact, we have shown in the above proof that
the conclusion of Theorem~\ref{thm:caracmu} holds true if one requires that the test functions
$\phi$ satisfy $\inf_{I\times B}\phi>0$ for {\em every ball} $B\Subset\O$, hence for every
compact set $K\subset\O$.
One could also check that the proof still works if
instead of requiring the test-functions to satisfy 
$\sup_{ I\times \O}\phi <\infty$ and $\inf_{I\times B}\phi>0$,
one asks the weaker condition
$$\begin{array}{lll}&\displaystyle\liminf_{t\to +\infty}
\left(\inf_{s,s+t\in I}\frac{\ln \|\phi (t+s,\cdot)\|_{L^{\infty}(\O)}- \ln \|\phi (s,\cdot)\|_{L^{\infty}(\O)}}{t}\right)&\geq 0 \\ 
&\\
\Big(\hbox{resp. } &\displaystyle\limsup_{t\to +\infty}
\left(\sup_{s,s+t\in I}\frac{\ln \|\phi (t+s,\cdot)\|_{L^{\infty}(\O)}- \ln \|\phi (s,\cdot)\|_{L^{\infty}(\O)}}{t}\right)&\leq 0 \Big)\\
\end{array}$$
in the definition of $\mu_{b,p}(I)$ (resp.  $\lambda_{b,p}(I)$).
This means that these changes do not alter the definitions of 
$\mu_{b,p}(I)$ and $\lambda_{b,p}(I)$.
\end{remark}

\begin{proof}[Proof of Corollary \ref{cor:caracmu}]
	By Theorem \ref{thm:caracmu}, and with the notation of its proof, one has
 $\mu_{b,p}(I)=-\lgr{\beta}_I$ and $\lambda_{b,p}(I)=-\ggr{\beta}_I$.
       The result then follows from Proposition~\ref{pro:gr-char}.
\end{proof}

\begin{proof}[ Proof of Theorem \ref{thm:caraceta}.]
	{\em The formula for $\lambda_b$.}\\ 
	In terms of the function $\beta$ defined by~\eqref{beta}, we need to show that
	\Fi{lambdab=}
	\lambda_b(\R)=\lambda_b(\R^+)= -\limsup_{t\to +\infty}\frac{\beta(t)}{t}.
	\Ff
	Firstly, let $\lambda\in\R$ be such that there exists a positive, bounded function 
  $\phi$ satisfying $P\phi\geq\lambda\phi$ in $\R^+\times\O$. 
Our aim is to use $\phi$ to control $u_P$ from above.
To do that, we need to show that (a large multiple of) $\phi$ controls $u_P$ at some time.
This would be immediate if we had Hopf's lemma and~$C^1$ regularity of $u_P$,
which is not our case due to the lack of regularity of $\O$ and of the coefficients of~$P$.
To circumvent such difficulty, we consider the solution to $Pu=0$ on $\R^+\times\O$
	under Dirichlet boundary condition and with initial datum 
	$\phi(0,\.)$. By comparison we get
	\Fi{psi<}
 \forall t\geq0,\quad
	\|u(t,\.)\|_{L^\infty(\O)}\leq
 \|\phi\|_{L^\infty(\R^+\times\O)} e^{-\lambda t}.
	\Ff
 Moreover, by Theorem~\ref{thm:Harnack}, there is
        a constant $C>0$ such that 
$$M:=\sup_{x\in \O}\frac{u_P(1,x)}{u(1,x)}\leq
C\inf_{x\in \O}\frac{u_P(1,x)}{u(1,x)}<+\infty.$$
Then, again by comparison, and using \eqref{psi<}, we infer the desired estimate
$$\forall t\geq1,\quad
\|u_P(t,\.)\|_{L^\infty(\O)}\leq M\|u(t,\.)\|_{L^\infty(\O)}
\leq M  \|\phi\|_{L^\infty(\R^+\times\O)} e^{-\lambda t},$$
which in turn yields
	$$\limsup_{t\to +\infty}\frac{{\beta}(t)}t\leq-\lambda.$$
Taking the supremum of $\lambda$ for which a function $\phi$ as above exists 
one gets 
	$$\limsup_{t\to +\infty}\frac{{\beta}(t)}t\leq-\lambda_b(\R^+).$$

	In order to get a lower bound for
 $\lambda_b(\R)$, take $\lambda<\lambda'$ such that
	$$\lambda'>-\liminf_{t\to -\infty}\frac{\beta(t)}{t},\qquad
 \lambda<-\limsup_{t\to +\infty}\frac{\beta(t)}{t}.$$
	Let $\gamma : \R \to \R$ be a smooth, nonincreasing function
	satisfying 
	$$\gamma (t)=\lambda'\ \text{ if }t\leq -1,\qquad
	\gamma (t)=\lambda\ \text{ if }t\geq 1.$$
	Define 
	$$\phi_b (t,x) := u_{P}(t,x) e^{\int_{0}^{t}\gamma (\tau)d\tau}.$$ 
	This function satisfies $P\phi_b = \gamma (t) \phi_b\geq \lambda\phi_b$ in $\R\times \O$. 
	Moreover, by our choices of $\lambda$ and $\lambda'$ we have, one one hand,
 for $t>1$,
	$$\|\phi_b (t,\cdot)\|_{L^{\infty}(\O)} = e^{\beta(t)+{\int_{0}^{1}\gamma(s)ds +
		 \lambda(t-1)}}  \to 0 \quad \hbox{ as } t\to +\infty,$$
   and on the other hand, for $t<-1$,
	$$\|\phi_b (t,\cdot)\|_{L^{\infty}(\O)} = 
 e^{\beta(t)+{\int_{0}^{-1}\gamma(s)ds +
		 \lambda'(t+1)}}  \to 0 \quad \hbox{ as } t\to -\infty.$$
 Thus, $\phi_b$ is bounded. This shows that 
	$\lambda_b(\R)\geq \lambda$ and thus, by the arbitrariness of $\lambda$,
	$$\lambda_b(\R) \geq -\limsup_{t\to +\infty}\frac{\beta(t)}t.$$
	
	Summing up, we have shown that
	$$\lambda_b(\R^+)\leq -\limsup_{t\to +\infty}\frac{\beta(t)}t\leq\lambda_b(\R).$$
	Since the inequality $\lambda_b(\R)\leq\lambda_b(\R^+)$ is an immediate consequence of the 
	definition,~\eqref{lambdab=} follows.
		

 \bigskip
	{\em The formula for $\mu_{p}(\R)$.}\\
 Let $\lambda\in\R$ be such that there exists a positive function 
  $\phi$ vanishing on $\R^+\times\partial\O$ and
  satisfying $P\phi\leq\lambda\phi$ in $\R^+\times\O$ and $\inf_{\R^+\times B}\phi\geq K>0$
  for some open ball $B\Subset \O$. Then, as before, considering 
  the solution to $Pu=0$ on $\R^+\times\O$
	under Dirichlet boundary condition and with initial datum 
	$\phi(0,\.)$, one finds
	$$
 \forall t\geq0,\ x\in\R^N,\quad
	u(t,x)\geq
 \phi(t,x) e^{-\lambda t}.
	$$
 Moreover, by the Harnack-type inequality Theorem~\ref{thm:Harnack}, there is
    $C>0$ such that 
$$m:=\sup_{x\in \O}\frac{u(1,x)}{u_P(1,x)}\leq
C\inf_{x\in \O}\frac{u(1,x)}{u_P(1,x)}<+\infty.$$
Then, by comparison we derive
$$\forall t\geq1,\quad
\|u_P(t,\.)\|_{L^\infty(\O)}\geq\frac1{m}\|u(t,\.)\|_{L^\infty(\O)}
\geq \frac1{m}  \|\phi(t,\.)\|_{L^\infty(\O)} e^{-\lambda t}
\geq \frac K m e^{-\lambda t},$$
whence
	$$\liminf_{t\to +\infty}\frac{{\beta}(t)}t\geq-\lambda.$$
Taking the infimum over $\lambda$ yields
	$$\liminf_{t\to +\infty}\frac{{\beta}(t)}t\geq-\mu_p(\R^+).$$

	Next, the inequality $\mu_{p}(\R)\leq -\liminf_{t\to +\infty}\frac{\beta(t)}t$  
 is obtained in a similar way as the inequality 
 $\lambda_b(\R) \geq -\limsup_{t\to +\infty}\frac{\beta(t)}t$.
 Namely, one considers 
 $$\phi_p (t,x) := u_{P}(t,x) e^{\int_{0}^{t}\gamma (\tau)d\tau},$$ 
 with $\gamma$ nondecreasing and coinciding with $\lambda'$ on $(-\infty,-1]$
 and with $\lambda$ on $[1,+\infty)$, where  
 	$$\lambda'<-\limsup_{t\to -\infty}\frac{\beta(t)}{t},\qquad
  \lambda>-\liminf_{t\to +\infty}\frac{\beta(t)}{t}.$$
 One checks that this choice of $\lambda,\lambda'$ implies $\|\phi_p(t,\.)\|_{L^\infty(\O)}\to+\infty$ as $t\to\pm\infty$,
 and the same is true for $\inf_B\phi_p(t,\.)$, for any given open ball
 $B\Subset \O$, thanks to the boundary Harnack inequality.
 One can then use $\phi_p$ in the definition of $\mu_p(\R)$ and infer that
 $\mu_p(\R)\leq\lambda$.
 Since this is true for any arbitrary
 $\lambda>-\liminf_{t\to +\infty}\frac{\beta(t)}{t}$, the desired inequality follows.
 %
	%
 \end{proof}

\begin{proof}[Proof of Theorem \ref{thm:caraclambda}.]
 Let $\lambda\in\R$ be such that there exists a positive function 
  $\phi$ 
  satisfying $P\phi\geq\lambda\phi$ in $\R^-\times\O$ and $\inf_{\R^-\times B}\phi\geq K>0$
  for some open ball $B\Subset \O$. For $s<0$, let $u^s$ be 
  the solution to $Pu=0$ on $(-s,0]\times\O$
	under Dirichlet boundary condition and with initial datum 
	$u^s(s,\.)=\phi(s,\.)$. 
  One gets by comparison
	\Fi{u^s<}
 \forall t\in[s,0], \ x\in\O,\quad
	u^s(t,x)\leq
 \phi(t,x) e^{-\lambda(t-s)}.
 \Ff
Moreover, using Theorem \ref{thm:Harnack} we infer
the existence of a positive constant $C$ such that
$$\forall s\leq-1,\qquad
M_s:=\sup_{x\in \O}\frac{u_P(s+1,x)}{u^s(s+1,x)}\leq
C\inf_{x\in \O}\frac{u_P(s+1,x)}{u^s(s+1,x)}\leq
\frac{\|u_P(s+1,\.)\|_{L^\infty(\O)}}{\inf_B u^s(s+1,\.)}.$$
The fact that $\inf_B u^s(s,\.)=\inf_B \phi(s,\.)\geq K$ implies that 
$\inf_B u^s(s+1,\.)\geq K'>0$ for some $K'$ independent of $s\leq-1$.
Hence there exists another constant $K''>0$ such that
$$\forall s\leq-1, \ x\in\O,\quad
u_P(s+1,x)\leq K''\|u_P(s+1,\.)\|_{L^\infty(\O)} u^s(s+1,x),$$
from which we deduce, by comparison,
$$\forall s\leq-1, \ t\in[s+1,0],\ x\in\O,\quad
u_P(t,x)\leq K''\|u_P(s+1,\.)\|_{L^\infty(\O)} u^s(t,x).$$
In particular, reclaiming \eqref{u^s<}   we get
$$u_P(0,x)\leq
K''\|u_P(s+1,\.)\|_{L^\infty(\O)}\phi(0,x) e^{\lambda s}.$$
It follows that, for all $s\leq-1$ and for given $x\in\O$,
$$\beta(s+1)\geq-\lambda s+\ln\frac{u_P(0,x)}{K''\phi(0,x)},$$
whence
$$-\limsup_{t\to -\infty}\frac{\beta(t)}{t}
\geq\lambda.$$
Taking the supremum in $\lambda$ we derive the above inequality 
with $\lambda$ replaced by $\lambda_p(\R^-)$.

Similarly, one derives the inequality
$$
-\liminf_{t\to -\infty}\frac{\beta(t)}{t}
\leq\mu_b(\R^-).
$$
Namely, starting from
$\lambda\in\R$ such that there exists a bounded, positive function 
  $\phi$ 
  satisfying $P\phi\leq\lambda\phi$ in $\R^-\times\O$ and vanishing on
  $\R^-\times\partial\O$, one considers, for $s<0$, 
  the solution to the Dirichlet problem $Pu^s=0$ on $(-s,0]\times\O$
	with initial datum 
	$u^s(s,\.)=\phi(s,\.)$. 
  One gets by comparison
$$
 \forall t\in[s,0], \ x\in\O,\quad
	 \phi(t,x) e^{-\lambda(t-s)}\leq u^s(t,x)\leq 
  \|\phi\|_\infty e^{\|c\|_\infty(t-s)}.
$$
Moreover, by Theorem \ref{thm:Harnack}, there exists
$C>0$ such that
$$\forall s\leq-1,\qquad
m_s:=\sup_{x\in \O}\frac{u^s(s+1,x)}{u_P(s+1,x)}\leq
C\inf_{x\in \O}\frac{u^s(s+1,x)}{u_P(s+1,x)}\leq
\frac{C\|\phi\|_\infty e^{\|c\|_\infty}}{\|u_P(s+1,\.)\|_{L^\infty(\O)}}\,.$$
Then, by the comparison principle, we deduce that
$$\forall s\leq-1, \ t\in[s+1,0],\ x\in\O,\quad
u^s(t,x)\leq 
\frac{C\|\phi\|_\infty e^{\|c\|_\infty}}{\|u_P(s+1,\.)\|_{L^\infty(\O)}}
\,u_P(t,x),$$
whence in particular
$$\forall s\leq-1, \ x\in\O,\quad
\phi(0,x) e^{\lambda s}\leq
\frac{C\|\phi\|_\infty e^{\|c\|_\infty}}{\|u_P(s+1,\.)\|_{L^\infty(\O)}}
\,u_P(0,x).$$
Applying this inequality for fixed $x\in\O$, and letting $s\to-\infty$,
one eventually gets
$$-\liminf_{t\to -\infty}\frac{\beta(t)}{t}
\leq\lambda.$$
Taking the infimum in $\lambda$ we derive the above inequality 
with $\lambda$ replaced by $\mu_b(\R^-)$.

Finally, the reverse inequalities for $\lambda_p(\R)$,
$\mu_b(\R)$ are obtained by considering 
the functions $\phi_b$, $\phi_p$ defined in the
proof of Theorem \ref{thm:caraceta}. They are respectively
bounded on $\R\times\O$ and strictly positive on 
$\R\times B$, for $B\Subset\O$, and moreover satisfy on
$\R\times \O$:
$P\phi_b \leq \lambda'\phi_b$ with given
$\lambda'>-\liminf_{t\to -\infty}\frac{\beta(t)}{t}$
and $P\phi_p \geq \lambda'\phi_p$ with given
$\lambda'<-\limsup_{t\to -\infty}\frac{\beta(t)}{t}$,
together with Dirichlet boundary conditions.
This yields the desired inequalities for $\lambda_p(\R)$,
$\mu_b(\R)$.
\end{proof}

\begin{remark}\label{rk:relax}
First, the function $\phi_p$ constructed in the proofs of
Theorems \ref{thm:caraceta}, \ref{thm:caraclambda} fulfills
$\inf_{\R\times K}\phi_p>0$ for every compact set $K\subset\O$.
This shows that the definitions of 
$\lambda_{p}(\R)$, $\mu_{p}(\R)$ do not change if one requires such a 
stronger condition on the test functions.
Second, one could check that the proofs of
Theorems \ref{thm:caraceta}, \ref{thm:caraclambda}
still work if the condition $\sup_{\R\times \O}\phi<\infty$
 is relaxed by
$$\limsup_{t\to +\infty}\frac{1}{t}\ln \|\phi(t,\cdot)\|_{L^\infty (\O)}\leq 0,$$ 
as well as if the condition $\inf_{\R\times B}\phi>0$ 
is relaxed by
$$\liminf_{t\to +\infty}\frac{1}{t}\ln \|\phi(t,\cdot)\|_{L^\infty (\O)}\geq 0.$$ 
As a consequence, these changes do not alter 
the definitions of $\lambda_{b}(\R)$, $\mu_{b}(\R)$  and
$\lambda_{p}(\R)$, $\mu_{p}(\R)$ respectively.
\end{remark}


\section{Completing the proofs}














\subsection{Proof of the comparison results between the generalized principal eigenvalues}

\begin{proof}[Proof of Proposition \ref{pro:basic}.]
	The equality $\lambda_b(\R^-)=+\infty$ simply follows by taking $\phi = e^{\gamma t}$, with $\gamma$ 
	arbitrarily large, in the definition. The identities $\lambda_p(\R^+)=+\infty$ and
$\mu_p(\R^-)=\mu_b(\R^+)=-\infty$ are proved similarly.

Let us prove the inequalities.
The lower bound for $\lambda_{b,p}(I)\geq-\sup c$ 
is simply obtained by taking a test-function $\phi$ which is constant in the definition of $\lambda_{b,p}(I)$.
The inequalities 
$\lambda_{b,p}(I)\leq \lambda_p(I)$, $\lambda_{b,p}(I)\leq\lambda_b(I)$ and
$\mu_b(I)\leq \mu_{b,p}(I)$ and $\mu_p(I)\leq \mu_{b,p}(I)$ are simple
consequences of the definitions. Finally, the inequalities
$\lambda_b(I)\leq\mu_p(I)$ and $\lambda_p(I)\leq\mu_b(I)$ are derived from
Theorems \ref{thm:caraceta} and \ref{thm:caraclambda} respectively.
This proves all the above inequalities.

Let us move on to the monotonicity property with respect to the domain.
Namely, we need to show that if $\O'\subset\O$ is another smooth domain,
then the associated notions of generalized principal eigenvalues in $I\times\O'$
are greater than or equal to the corresponding ones in $I\times\O$. This is straightforward in the case of 
the $\lambda$'s, since one can use the test functions in $I\times\O$ as test functions in $I\times\O'$ (up to choosing the ball $B$ inside $\O'$, which is possible thanks to the 
Remark~\ref{rk:ball}$(\ref{rk:ballB})$).
This cannot be done for the $\mu$'s because test functions must satisfy the
Dirichlet boundary condition. However, one~can argue the other way around:
starting from a subsolution in $I\times\O'$ which vanishes on $I\times\partial\O'$,
one can consider its extension to $0$ outside $I\times\O'$,
which turns out to be a {\em generalized} subsolution~in~$I\times\O$
and can be used to derive the desired inequality.

Let us describe it in detail in the case of the notion $\mu_{b,p}$, 
the other cases being analogous. Let us make the dependence
on the domain by writing $\mu_{b,p}(I\times\O)$ and $\mu_{b,p}(I\times\O')$.
One verifies that in 
the derivation of the inequality $\mu_{b,p}(I\times\O)\geq-\lgr{\beta}_I$
in the proof of Theorem \ref{thm:caracmu}, 
the properties of the test function $\phi$ associated with $\lambda>\mu_{b,p}(I\times\O)$
have only been used to obtain a solution $u$ that satisfies 
the two-sided estimate~\eqref{<psi<}.
As a matter of fact, if~$\phi$ is now associated with $\lambda>\mu_{b,p}(I\times\O')$,
one may extend it by $0$ outside $I\times\O'$, getting
a generalized subsolution 
in $I\times\O$. Thus 
one derives~\eqref{<psi<} exactly as before, using the comparison principle.
From that, continuing the proof,
one eventually infers $\mu_{b,p}(I\times\O')\geq-\lgr{\beta}_I$, where $\beta$ 
corresponds to the function $u_P$ in the domain $\O$ (and not in~$\O'$).
But then we conclude by Theorem \ref{thm:caracmu} that 
$\mu_{b,p}(I\times\O')\geq \mu_{b,p}(I\times\O)$.

It remains to prove the well-posedness of $\mu_{b,p}$, that is, that the set 
in its definition is nonempty. Since $\mu_{b,p}$ is the greatest among all the notions,
thanks to the inequalities showed above, this will imply the finiteness of all the notions.
By monotonicity, it is sufficient to prove it in a subdomain of~$\O$.
Consider the ball $B$ in the definition of $\mu_{b,p}(I\times\O)$; 
let us assume without loss of generality that it is centered at the origin,
and let $r$ be its radius. Then consider a function
$\phi(t,x):=\chi(|x|)$, where $\chi$ is an even, smooth function satisfying
$$\chi>0\text{ \ in }(-r,r),\qquad
\chi(r)=\chi'(r)=0,\qquad\chi''(r)>0.$$
On the one hand, 
by the uniform parabolicity of $P$ and by the boundedness of its coefficients,
the function $\phi$ satisfies $P\phi\geq0$ in some neighborhood $I\times U$
of $I\times\partial B$. On the other hand, $P\phi$ is bounded and one has
$$\inf_{I\times(B\setminus U)}\phi>0,$$ hence one can find $\lambda>0$ such that
$$P\phi\leq \lambda\phi \quad\text{in }I\times(B\setminus U).$$
This gives $\mu_{b,p}(I\times B)\leq\lambda$, which concludes the proof.
\end{proof}


\subsection{Proofs of the applications to the semilinear problem}

\begin{proof}[Proof of Proposition \ref{prop:initdatum}.]
Let $u$ be as in the statement of the proposition.
We start with observing that, as $M$ is a supersolution of \eqref{RD} by \eqref{f<0}, 
it follows from the comparison principle that
$u(t,x)\leq\max\{M,\|u_0\|_\infty\}$ for $t\geq0$, $x\in\O$. 

Assume that $\mu_{b,p}(\R^+)<0$.
        There exists then $\lambda<0$ and $\phi$ such~that	
        $$\phi>0\quad\text{and}\quad
	P\phi\leq\lambda\phi\quad\text{in }\R^+\times\O,$$
	and moreover
 $$\phi=0\text{ on }\R^+\times\partial\O,\qquad
	\inf_{t\in \R^+,\ x\in B}\phi(t,x)\geq \frac1K \quad\text{and}\quad
 \sup_{t\in \R^+,\ x\in \O}\phi(t,x)\leq K,$$ 
 for some open ball $B\Subset \O$ and some $K>0$.
 In order to compare $u$ with a suitable multiple of $\phi$, we 
 preliminarily need to show that the comparison holds at some given time, say $t=1$, that is
 \Fi{phi<v}
 \forall x\in\O,\quad
 m\phi(1,x)\leq u(1,x).
 \Ff
 Since, we cannot directly infer this from the Hopf lemma because of the lack of regularity of $\phi$,
 we will rather employ the Harnack inequality.
 To start with, we write the equation for $u$ in linear form: 
 \Fi{linearized}
 \partial_t u-a_{ij}(t,x)\partial_{ij} u-b_i(t,x)\partial_i u= c(t,x)u,
 \Ff
 with $c(t,x):=f(t,x,u(t,x))/u(t,x)$ if $u(t,x)\neq0$.
 Then we let $v$ be the solution to \eqref{linearized},
	under Dirichlet boundary condition and with initial datum~$\phi(0,\.)$.
On the one hand, the Harnack-type inequality Theorem~\ref{thm:Harnack} yields the existence of a constant
    $C>0$ such that 
$$m':=\sup_{x\in \O}\frac{v(1,x)}{u(1,x)}\leq
C\inf_{x\in \O}\frac{v(1,x)}{u(1,x)}<+\infty.$$
On the other hand, since $u$ vanishes on $\partial\O$ and it 
 is uniformly continuous, by parabolic regularity, there exists a set $K\Subset\O$
 such that $c(t,x)\geq f_s'(t,x,0)+\lambda$ on $\O\setminus K$.    
It follows that $\phi$ is a subsolution to \eqref{linearized} on $\O\setminus K$.
Take $m''\in(0,1)$ small enough so that $v\geq m''\phi$ on $[0,1]\times K$. 
Then the comparison principle yields $v\geq m''\phi$ on $[0,1]\times(\O\setminus K)$,
and thus on $[0,1]\times\O$.
Gathering together the two inequalities obtained above 
we get $m'u(1,x)\geq v(1,x)\geq m''\phi(1,x)$ for $x\in\O$, i.e.~\eqref{phi<v}

 %
%
Now, since $P\phi\leq\lambda\phi$, with $\lambda<0$,
by the regularity of $f$ there exists $\sigma>0$ such that
$f(t,x,s)\geq (f_s(t,x,0)+\lambda) s$ for $(t,x)\in\R^+\times\R^N$ and $s\in[0,\sigma)$.
Therefore,
$m \phi$ is a subsolution to the Dirichlet problem \eqref{RD}, up to decreasing 
$m$ if need be.
We can then
compare $m\phi$ with $u$ and derive
 $$\forall t\geq1,\ x\in\O,\quad
 m \phi(t,x)\leq u(t,x).$$
In particular, one finds $u(t,x)\geq m/K$ for $t>1$ and $x\in B$.
Then the desired lower bound follows from the standard interior Harnack inequality
\cite{KS80}, considering again $u$ as a solution of \eqref{linearized}.

Let us show the reverse implication. The KPP hypothesis 
(\ref{hyp:fKPP}) implies $Pu\leq 0$ on $\R^+\times \O$. 
Suppose that $u$ fulfills the persistence property \eqref{eq:persistence}. 
Then, thanks to the Harnack inequality (for the linear equation \eqref{linearized})
for any set $K\Subset\O$, one can find some $T>0$
such that $\inf_{[T,+\infty)\times K}u>0$.
This immediately gives $\mu_{b,p}(\R^+)\leq0$, owing to~Remark~\ref{rk:ball}$(\ref{rk:IT})$.
However, we need to rule out the possibility $\mu_{b,p}(\R^+)=0$;
this is the only point where the strict inequality in (\ref{hyp:fKPP}) is needed.
To achieve our goal, we perturb $u$ by a positive sub-barrier,
which is not completely standard due to the temporal dependence of the operator.
For given $\e>0$, we call
$$\phi(t,x):=u(t,x)+\e\vp^2(x),$$
where $\vp$ is the solution to 
$$-\Delta \vp=1\quad\text{in }\,\O,\qquad \vp=0\quad\text{on }\,\partial\O.$$
Since $\O$ is of class $C^{2,\alpha}$, it follows from the standard elliptic theory 
(see e.g.~\cite{GT}) that 
$\vp\in C^{2,\alpha}(\ol\O)$ and moreover $\vp$ is positive in $\O$
and satisfies $\partial_\nu\vp>0$ on $\partial\O$
owing to Hopf's maximum principle.
Direct inspection~reveals
$$P(\vp^2)=-2\vp a_{ij}(t,x)\partial_{ij}\vp 
-2a_{ij}(t,x)\partial_i\vp\partial_j\vp-2\vp b_i(t,x)\partial_i \vp- f_{s}'(t,x,0)\vp^2.$$
%
%
%
%
Using the ellipticity of $(a_{ij})$ and the fact that $\vp=0$ and $|\nabla\vp|\neq0$ on $\partial\O$,
one can then find some $h>0$ and $K\Subset \O$ such that $P(\vp^2)<-h$ in $\O\setminus K$.
For $t>0$ and $x\in\O\setminus K$ we find
$$P\phi=Pu+\e P(\vp^2)=f(t,x,u)-f_s'(t,x,0)u+\e P(\vp^2)\leq
-\e h,$$
where we have only used the large inequality in \eqref{hyp:fKPP}.
Recalling that $u$ is bounded, we then have $P\phi\leq\lambda\phi$ in $\R^+\times(\O\setminus K)$,
for some $\lambda'<0$ with $|\lambda'|$ sufficiently small.
Instead, in $K$, we use the stronger hypothesis \eqref{hyp:fKPP}.
The uniform continuity of $u$ and the persistence property
imply that there exist $T,\delta>0$ such that $u>\delta$ on $[T,+\infty)\times K$. 
Then, since $f$ is locally continuous in $s$, uniformly with respect to $t,x$, 
property \eqref{hyp:fKPP} yields
$$\zeta:=\inf_{[T,+\infty)\times K}\big(f_s'(t,x,0)u-f(t,x,u)\big)>0.$$
We then get in $[T,+\infty)\times K$,
$$P\phi\leq -\zeta+\e P(\vp^2),$$
which, for $\e$ small enough, gives $P\phi\leq\lambda''\phi$ for a suitable $\lambda''<0$.
In conclusion, we have shown that for $\e$ small, 
$P\phi\leq\lambda\phi$ in $[T,+\infty)\times \O$
with $\lambda:=\max\{\lambda',\lambda''\}<0$.
This implies that $\mu_{b,p}((T,+\infty))<0$. 
Since $\mu_{b,p}((T,+\infty))=\mu_{b,p}(\R^+)$ 
(cf.~Remark~\ref{rk:ball}$(\ref{rk:IT})$), the proof is concluded.
%
%
%
\end{proof}

\begin{proof}[Proof of Proposition \ref{pro:RD-existence}.]
Suppose that \eqref{RD}
admits a bounded entire solution $u$ satisfying \eqref{infK}.
Then, using exactly the same function $\phi(t,x):=u(t,x)+\e\vp^2(x)$ as in the proof of 
Proposition \ref{prop:initdatum} yields
${\mu_{b,p}(\R)<0}$. 

Suppose now that $\mu_{b,p}(\R)<0$.
There exists then $\lambda<0$, $\phi>0$ in $\R\times\O$ and $B\Subset\O$ satisfying
$$\sup_{\R\times\O}\phi<+\infty,\quad \phi=0\text{ on }\R\times\partial\O,
\quad  \inf_{\R\times B}\phi>0,\quad
P\phi\leq\lambda\phi\text{ in }\R\times\O.$$
The uniform regularity of $f$ implies that 
$f(t,x,m\phi)\geq (f_s(t,x,0)+\lambda)m\phi$ for all $(t,x)\in\R^{N+1}$, provided that
$m>0$ is sufficiently small. We choose $m>0$ small enough in such a way that 
this inequality holds and, in addition, $m\phi\leq M$, where $M$ is from \eqref{f<0}. Hence, 
the function $m\phi$ is a subsolution of~\eqref{RD}
which is smaller than the constant function $M$, the latter being
a supersolution of~\eqref{RD}.
Let~$u_n$ be the solution of~\eqref{RD} in $(-n,+\infty)\times\O$ 
with initial condition $u_n(-n,\.)=M$. 
It follows from the standard comparison principle that $m\phi\leq u_n\leq M$ 
in $(-n,+\infty)\times\O$.
By parabolic estimates up to the boundary, the sequence $\seq{u}$ converges (up to subsequences)
locally uniformly in $\R\times\ol\O$ to a solution $m\phi\leq u\leq M$ of 
\eqref{RD}. One infers in particular that 
that $\inf_{\R\times B}u\geq m\inf_{\R\times B}\phi>0$. 
%
As usual, the interior
Harnack inequality implies that $u$ satisfies also condition \eqref{infK}.
\end{proof}

\begin{remark} 
The reader could easily adapt our proof in order to show that
if one drops condition~(\ref{infK}), then the relevant notion for the
existence result is $\mu_{b}(\R)$ rather than $\mu_{b,p}(\R)$.
Namely, problem \eqref{RD} admits a positive bounded entire solution if $\mu_{b}(\R)<0$ 
		and only if $\mu_{b}(\R)\leq 0$. 
\end{remark}

\begin{proof}[Proof of Theorem \ref{thm:RD-uniqueness}.]
Suppose that the problem \eqref{RD} admits
two bounded ancient solutions~$u$ and~$v$ satisfying \eqref{infK}. 
Let us write the equation for $v$ in linear form, namely
$$\partial_t v-a_{ij}(t,x)\partial_{ij} v-b_i(t,x)\partial_i v= c(t,x)v,$$
 with $c(t,x):=f(t,x,v(t,x))/v(t,x)$. Observe that the function $c$ is bounded, namely,
$$|c|\leq L:=\|f_s'\|_\infty.$$
 Fix $T\leq-1$. Let $\t u$ be the solution of the above equation
for $t\in(T,0]$, $x\in\O$, under Dirichlet 
boundary condition, with initial datum $\t u(T,\.)=u(T,\.)$.
It then follows from the Harnack inequality, Theorem~\ref{thm:Harnack},
that there exists a constant $C>0$ independent of $T$ such that 
$$\forall t\in[T+1,0],\quad
\sup_{x\in \O}\frac{\t u(t,x)}{v(t,x)}\leq
C\inf_{x\in \O}\frac{\t u(t,x)}{v(t,x)}
\leq C\frac{\|\t u(t,\.)\|_\infty}{\|v(t,\.)\|_\infty}.$$
On the one hand, the term $\|v(t,\.)\|_\infty$ is uniformly bounded from below away from zero,
because $v$ satisfies \eqref{infK}.
On the other hand, the function
$\|u\|_\infty e^{L(t-T)}$
is a supersolution of the equation satisfied by $\t u$ and therefore, one deduces 
from the comparison principle that
$\|\t u(t,\.)\|_\infty\leq \|u\|_\infty e^{L(t-T)}$ for $t\in(T,0]$.
As a consequence, there exists another constant $C'$ independent of $T$ such that
$$\sup_{x\in \O}\frac{\t u(T+1,x)}{v(T+1,x)}
\leq C' .$$
Next, we have that the function $\hat u(t,x):=\t u(t,x)e^{2L(t-T)}$ satisfies
$$\partial_t \hat u-a_{ij}(t,x)\partial_{ij} \hat u-b_i(t,x)\partial_i \hat u= 
[c(t,x)+2L]\hat u.$$  
Notice that
$$c(t,x)+2L=\frac{f(t,x,v(t,x))}{v(t,x)}+2L\geq \frac{f(t,x,u(t,x))}{u(t,x)}.$$
The last term is the zero order coefficient of the equation satisfied by $u$,
written in linear form.
Hence, the comparison principle yields $\hat u\geq u$ on $[T,0]\times\O$. 
Gathering together the above estimates, one ends up with
$$\sup_{x\in \O}\frac{u(T+1,x)}{v(T+1,x)}\leq
\sup_{x\in \O}\frac{\t u(T+1,x)}{v(T+1,x)}\, e^{2L}
\leq C'e^{2L}.$$
We emphasize that this estimate holds for all $T\leq-1$, with $C'$ independent of $T$.
We can then define
$$\bar k:=\inf\{h>0\,:\, kv>u\ \text{ in } \R^-\times\O\}.$$
There holds that $\bar k>0$ and that the function $w:= \bar kv-u$ is nonnegative.

We want to show that $\bar k\leq1$. Assume by way of contradiction that 
this is not the case.
Using the hypothesis \eqref{fconcave} we get
$$\partial_tw-a_{ij}(t,x)\partial_{ij}w-b_i(t,x)\partial_iw=\bar k f(t,x,v)-f(t,x,u)
> f(t,x,\bar kv)-f(t,x,u).$$
It follows from the strong maximum principle that $w$ is a positive supersolution
of the equation
$$\partial_t\t w-a_{ij}(t,x)\partial_{ij}\t w-b_i(t,x)\partial_i\t w=\bar c(t,x)\t w,$$
with $\bar c(t,x):=(f(t,x,\bar kv)-f(t,x,u))/w(t,x)$,
which is bounded.
We now fix $T\leq-1$ and
consider the solutions $\t u$ and $\t w$ of the above linear equation on $(T,0]\times\O$,
under Dirichlet boundary condition, and with initial data $\t u(T,\.)=u(T,\.)$ and $\t w(T,\.)=w(T,\.)$.
Repeating the arguments of the first part of the proof, we end up with the following estimate:
$$\sup_{x\in \O}\frac{u(T+1,x)}{\t w(T+1,x)}\leq
\sup_{x\in \O}\frac{\t u(T+1,x)}{\t w(T+1,x)}\, e^{2L}\leq
 C\frac{\|u\|_\infty}{\|\t w(T+1,\.)\|_\infty}\, e^{3L}.$$
Then, as the parabolic comparison principle yields $w\geq \t w$ on $(T,0]\times\O$, 
one derives
$$\sup_{x\in \O}\frac{u(T+1,x)}{w(T+1,x)}\leq
 C\frac{\|u\|_\infty}{\|\t w(T+1,\.)\|_\infty}\, e^{3L}.$$
 
Suppose for a moment that $w$ satisfies \eqref{infK}.
Then, it is easily seen, for instance by contradiction, 
that there exists a constant $C'$ independent of $T$ such that
$$
\forall K\Subset\O,\quad\min_{K}\t w(T+1,\.)\geq C'.$$
As a consequence, there exists $C''>0$ such that
$$\sup_{\R^-\times\O}\frac{u}{w}\leq C'',$$
which, recalling the definition of $w$, rewrites as
$\bar kv-u\geq u/C''$, that is, $\bar kv\geq (1+1/C'')u$.
This contradicts the definition of $\bar k$.


We have therefore shown that $w$ cannot fulfill \eqref{infK}, thus 
one can find a compact set $K\subset\O$ and
a sequence $((t_n,x_n))_{n\in\N}$ in $\R^-\times K$
such that
\Fi{snyn}
w(t_n,x_n)\to0,\quad x_n\to z\in K\qquad\text{as }n\to+\infty.
\Ff 
%
In order to get a contradiction, we cannot directly pass to the limit in the equation
satisfied by
$w(\.+t_n,\.)$, because we do not have enough regularity on the coefficients.
To circumvent this difficulty, 
we apply once again the interior Harnack inequality and infer that, for a given 
open ball $B\Subset\O$
containing the point $z$, there exists a constant $C>0$ such that
$$\sup_{[t_n-3,t_n-1]\times B}w\to0\quad\text{as }\;n\to+\infty.$$
Then, we use the interior parabolic estimates (see e.g.\ \cite[Theorem 7.22]{Lie},
this is where the uniform continuity of the $a_{ij}$ is required)
and infer that 
$$\|w\|_{W^{1,2}_{p}([t_n-2,t_n-1]\times B')}\to0\quad\text{as }\;n\to+\infty,$$
for any given ball $B'\Subset B$.
Thus, integrating the equation satisfied by $w$ yields
$$\limn\int_{[t_n-2,t_n-1]\times B'}\big(\bar k f(t,x,v)-f(t,x,u)\big)
=0.$$ 
Recalling the definition of the function $\bar c$, 
we rewrite the term inside the above integral as
$$\bar k f(t,x,v)-f(t,x,u)=
\bar k f(t,x,v)-f(t,x,\bar kv)+\bar c(t,x)w(t,x),$$
and we find
$$\limn\int_{[t_n-2,t_n-1]\times B'}\big(\bar k f(t,x,v)-f(t,x,\bar kv)\big)
=0.$$ 
Recall however that $\inf_{\R^-\times B'}v>0$ because $v$ fulfills \eqref{infK},
and moreover it holds that
$$\bar k f(t,x,v)-f(t,x,\bar kv)=
\bar k v\left(\frac{f(t,x,v)}{v}-\frac{f(t,x,\bar kv)}{\bar k v}\right),$$
which is then bounded from below by a positive constant for $t\in\R^-$ and $x\in B'$,
thanks to the hypothesis~\eqref{fconcave} and the continuity of 
$s\mapsto f(t,x,s)$, which is uniform with respect to $t,x$.
We have reached a contradiction.

In conclusion, we have shown that $\bar k\leq1$, that is, $u\leq v$.
%
%
Switching the roles of $u$ and $v$ we derive $v\leq u$,
concluding the proof of the theorem.
\end{proof}


\subsection{Proofs of the applications to the linear problem}

\begin{proof}[Proof of Proposition \ref{prop:carac}.]
Let us call for short $u(t,x):=u(t,x;u_0)$.
By Theorem~\ref{thm:Huska} we know that there exist
$C,\gamma>0$ and $q\in\R$ such that
$$q - C\|u_{0}-q u_{P}(0,\.)\|_\infty\,e^{-\gamma t}\leq
\frac{\|u(t,\.)\|_\infty}{\| u_P(t,\.)\|_\infty}
\leq q + C\|u_{0}-q u_{P}(0,\.)\|_\infty
\,e^{-\gamma t}.$$
We actually know from the last assertion of the theorem that $q>0$, because $u(t,x)>0$ for $t>0$, $x\in\O$ 
due to the strong maximum principle.
As a consequence, we derive
$$\lim_{t\to+\infty}\frac1t\big(\ln\|u(t,\.)\|_\infty
-\ln\|u_P(t,\.)\|_\infty\big)=0.$$ 
We then deduce from Theorem \ref{thm:caraceta} that 
$$
\liminf_{t\to +\infty}\frac{\ln \|u(t,\.)\|_\infty}{t}=-\mu_{p}(\R^+),\qquad
\limsup_{t\to +\infty}\frac{\ln \|u(t,\.)\|_\infty}{t} =-\lambda_b(\R^+).$$
We then conclude using the boundary Harnack inequality, which yields,
 for given $x\in\O$,  the existence of a constant $C'>0$ such that
$$\forall t>2,\quad
\frac1{C'}
\|u(t-1,\cdot)\|_\infty\leq u(t,x)\leq \|u(t,\cdot)\|_\infty.$$
\end{proof}

\begin{proof}[Proof of Proposition \ref{prop:u^s}.]
We write for short $u^s(t,x)$ in place of $u^s(t,x;u_0)$.
By Theorem \ref{thm:Harnack} there exists a positive constant $C$ such that,
for any $s\in\R$ and $x\in\O$, there holds
$$
\frac1C\frac{\|u^s(s+1,\.)\|_\infty}{\| u_P(s+1,\.)\|_\infty}
\leq
\frac1C \sup_{\O}\frac{u^s(s+1,\.)}{u_P(s+1,\.)}
\leq\frac{u^s(s+1,x)}{u_P(s+1,x)}\leq C \inf_{\O}\frac{u^s(s+1,\.)}{u_P(s+1,\.)}\leq
C\frac{\|u^s(s+1,\.)\|_\infty}{\| u_P(s+1,\.)\|_\infty}.$$
One has by comparison that 
$$\|u^s(s+1,\.)\|_\infty\leq \|u_0\|_\infty e^{\|c\|_\infty}.$$
Furthermore, using that $u_0$ is continuous and strictly positive somewhere,
one easily sees that there exists a constant $K>0$
 independent of $s$ such that $\|u^s(s+1,\.)\|_\infty\geq K$.
 Summing up, there exists $K'>0$ such that
 $$\forall s\in\R, \ x\in\O,\quad
 \frac{1}{K'\| u_P(s+1,\.)\|_\infty}u_P(s+1,x)
\leq u^s(s+1,x)\leq \frac{K'}{\| u_P(s+1,\.)\|_\infty}u_P(s+1,x).$$
Therefore, the comparison principle eventually yields
 $$\forall s\in\R, \ x\in\O, \ t\geq s+1,\quad
 \frac{1}{K'\| u_P(s+1,\.)\|_\infty}u_P(t,x)
\leq u^s(t,x)\leq \frac{K'}{\| u_P(s+1,\.)\|_\infty}u_P(t,x).$$
For $s\leq-1$, we apply these estimates at $t=0$ and get
 $$\forall s\leq-1, \ x\in\O,\quad
 \frac{1}{K'\| u_P(s+1,\.)\|_\infty}u_P(0,x)
\leq u^s(0,x)\leq \frac{K'}{\| u_P(s+1,\.)\|_\infty}u_P(0,x).$$
Finally, we derive
$$\forall s\leq-1, \ x\in\O,\quad
\big|\ln u^s(0,x)+
 \ln \| u_P(s+1,\.)\|_\infty-\ln u_P(0,x)\big|\leq 
 \ln K'.$$
 The results then follow from Theorem \ref{thm:caraclambda}.
\end{proof}

\begin{proof}[Proof of Theorem \ref{thm:MP}.]
	The fact that that the validity of the {\em MP} implies $\mu_b(\R^-)\geq0$
 is straightforward, because $\mu_b(\R^-)<0$
	implies the existence of an ancient, bounded, 
 subsolution $\phi$ which violates the {\em MP}.
 
	Let us turn to the reverse implication.
	Suppose that there exists an ancient, bounded, 
 subsolution~$u$ to \eqref{Pu=0} that violates the~{\em MP}, that is, such that 
 $u(\bar t,\bar x)>0$ for some $\bar t<0$ and $\bar x\in\O$.
 It follows from the comparison principle that $\max_\O u(t,\.)>0$ for any $t<\bar t$.
        For $T<\bar t$, let $v_T$ be the solution to the problem \eqref{Pu=0} for $t>T$,
        with initial datum 
        $v_T(T,\.)=\max\{u(T,\.),0\}$.
        Then, by comparison, we get
        $$\forall t\geq T,\ x\in\O,\quad
        u(t,x)\leq v_T(t,x)\leq \|u\|_{L^\infty(\R^-\times\O)}e^{\|c\|_\infty(t-T)},$$
        and moreover $v_T(t,x)>0$ for $t>T$ and $x\in\O$.
        Let us call $K:=\|u\|_{L^\infty(\R^-\times\O)}e^{\|c\|_\infty}$.
Then, by Theorem \ref{thm:Harnack}, there exists a positive constant $C$ such that,
for any $T<0$, it holds
 $$\forall x\in\O,\quad
\frac{ v_T(T+1,x)}{u_P(T+1,x)}\leq C
\inf_{\O}\frac{ v_T(T+1,\.)}{u_P(T+1,\.)}\leq
C\frac{K}{\| u_P(T+1,\.)\|_\infty}.$$
It then follows, again by comparison, that, for any $T<0$,
$$\forall t\in[T+1,0),\ x\in\O,\quad
u(t,x)\leq v_T(t,x)\leq
C\frac{K}{\| u_P(T+1,\.)\|_\infty}u_P(t,x).$$
But, if by contradiction one had $\mu_b(\R^-)>0$, then 
Theorem \ref{thm:caraclambda} would imply the existence of
a sequence $\seq{T}$ diverging to $-\infty$ such that 
$\| u_P(T_n+1,\.)\|_\infty\to+\infty$ as $n\to+\infty$,
and therefore the above inequality would imply
$u(t,x)\leq0$ for all $t<0$, $x\in\O$. This is a contradiction.
\end{proof}


\subsection{Proofs of the results on  limit operators}\label{sec:proofopelim}

\begin{proof}[Proof of Theorem \ref{P*}.] We give the proof of the identities
stated in the theorem in the case of $\R^+$. The case of $\R^-$ follows from the 
same arguments, with minor modifications.

Consider an arbitrary limit operator $P^*\in \omega_{\R^+}(P)$ and let
$(t_{n})_{n\in\N}$ in $\R^+$ be the associated sequence. By parabolic estimates, the functions $u_P(\.+t_n,\.)/u_P(t_n,\.)$ are bounded in $W^{1,2}_{p,loc}(I\times\ol\O)$ for any $p\in (N+1,\infty)$. Hence, 
up to subsequences, these functions  converge weakly-$\star$ in $W^{1,2}_{p,loc}(I\times\ol\O)$ for any $p\in (N+1,\infty)$, and, using the Aubin-Lions lemma, strongly in $W^{0,1}_{p,loc}(I\times\ol\O)$ as $n\to+\infty$, to a function $u$. We could thus pass to the limit in the equation and get that $u$ is a positive solution to $P^* u=0$ in $\R\times\O$, vanishes on
$\R\times\partial\O$, and satisfies
$\|u(0,\.)\|_{L^\infty(\O)}=1$, namely, $u=u_{P^*}$ given by
Theorem \ref{thm:uP}.
It follows from Theorem \ref{thm:caraceta} that, on the one hand,
\[\begin{split}
-\mu_p(P^*,\R^+) &=
\liminf_{t\to+\infty}\frac1t \ln \| u_{P^*} (t,\cdot)\|_{L^\infty(\O)}\\
&
=\liminf_{t\to+\infty}\Big(\limn \frac1t\big(\ln \| u_P(t_n+t,\cdot)\|_{L^\infty(\O)}-
\ln \| u_P(t_n,\cdot)\|_{L^\infty(\O)}\big)\Big),
\end{split}\]
and the last term is greater than or equal to $-\mu_{b,p}(P,\R^+)$ 
by Theorem \ref{thm:caracmu}.
On the other hand, from Theorems \ref{thm:caraceta} and \ref{thm:caracmu},
we infer
\[-\lambda_b(P^*,\R^+) 
=\limsup_{t\to+\infty}\Big(\limn \frac1t\big(\ln \| u_P(t_n+t,\cdot)\|_{L^\infty(\O)}-
\ln \| u_P(t_n,\cdot)\|_{L^\infty(\O)}\big)\Big)
\leq-\lambda_{b,p}(P,\R^+).
\]
We have thereby shown that
\[\mu_{b,p}(P,\R^+)\geq\sup_{P^*\in\o_{\R^+}(P)}\mu_p(P^*,\R^+)
\quad\text{ and }\quad
\lambda_{b,p}(P,\R^+)\leq\inf_{P^*\in\o_{\R^+}(P)}\lambda_b(P^*,\R^+).\]

We know from Theorem \ref{thm:caraceta} 
that $\mu_p(P^*,\R^+)\geq\lambda_b(P^*,\R^+)$, for any operator $P^*$, therefore 
to conclude the proof we need to find some limit operators $P_1^*$ and $P_2^*$
for which $\mu_{b,p}(P,\R^+)=\lambda_b(P^*_1,\R^+)$ and
$\lambda_{b,p}(P,\R^+)=\mu_p(P^*_2,\R^+)$.

Let $I=\R^-$ or $I=\R^+$.
We claim that for all $n\in \mathbb{N}$, there exists $s_{n}\in I$ 
satisfying $s_{n}+n\in I$, with the following~property: 
\Fi{sn}
\forall t\in [1,n],\quad
\frac{1}{t} \Big(\ln \|u_{P} (s_{n}+t,\cdot)\|_{L^{\infty}(\O)} - 
\ln \| u_{P} (s_{n},\cdot)\|_{L^\infty(\O)} \Big)\leq 
- \mu_{b,p}(P,I)+\frac{1}{n}.
\Ff
The proof of this claim is similar to that of \cite[Proposition~4.4]{NR2}. 
Assume that it fails for some $n\in \mathbb{N}$. 
By Theorem \ref{thm:caracmu}, for $K\in \N$ large enough, 
there exists $\tau \in I$, depending on $K$,  such that $\tau+Kn \in I$ and 
\Fi{Kn}
\frac{1}{Kn}\Big(\ln \|u_{P} (\tau +Kn,\cdot)\|_{L^{\infty}(\O)} - 
\ln \| u_{P} (\tau,\cdot)\|_{L^\infty(\O)} \Big) <- \mu_{b,p}(P,I)
+\frac{1}{2n}.
\Ff
On the other hand, by the contradictory assumption,
there exists  $t_{0}\in [1,n]$ such that $\tau+t_{0}\in I$ and 
$$\frac{1}{t_{0}}\Big(\ln \|u_{P} (\tau+t_{0},\cdot)\|_{L^{\infty}(\O)} - 
\ln \| u_{P} (\tau,\cdot)\|_{L^\infty(\O)} \Big)>
- \mu_{b,p}(P,I)+\frac{1}{n}.$$
Then, there is $t_{1}\in [1,n]$ such that $\tau+t_{0}+t_{1}\in I$ and 
$$\frac{1}{t_{1}}\Big(\ln \|u_{P} (\tau+t_{0}+t_{1},\cdot)\|_{L^{\infty}(\O)} - 
\ln \| u_{P} (\tau+t_{0},\cdot)\|_{L^\infty(\O)}\Big) >
- \mu_{b,p}(P,I)+\frac{1}{n}$$
and thus:
$$n\Big(\ln \|u_{P} (\tau+t_{0}+t_{1},\cdot)\|_{L^{\infty}(\O)} - 
\ln \| u_{P} (\tau,\cdot)\|_{L^\infty(\O)} \Big)>
- n (t_{0}+t_{1})\mu_{b,p}(P,I)+t_{0}+t_{1}.$$
We continue until we find $\bar t\in [1,n]$ such that 
$$n\Big(\ln \|u_{P} (\tau+Kn+\bar t,\cdot)\|_{L^{\infty}(\O)} - 
\ln \| u_{P} (\tau,\cdot)\|_{L^\infty(\O)}\Big) >
- n(Kn+\bar t)\mu_{b,p}(P,I)+Kn+\bar t.$$
It follows that 
$$\begin{array}{l} 
\ln \|u_{P} (\tau+Kn,\cdot)\|_{L^{\infty}(\O)} - 
\ln \| u_{P} (\tau,\cdot)\|_{L^\infty(\O)} \\
>\ln \|u_{P} (\tau+Kn,\cdot)\|_{L^{\infty}(\O)} - 
\ln \| u_{P} (\tau+Kn+\bar t,\cdot)\|_{L^\infty(\O)}
- (Kn+\bar t)\mu_{b,p}(P,I)+K+\bar t/n.
\end{array}$$
Then, using the fact that 
$$\| u_{P} (\tau+Kn+\bar t,\cdot)\|_{L^\infty(\O)}
\leq \| u_{P} (\tau+Kn,\cdot)\|_{L^\infty(\O)}e^{\|c\|_\infty \bar t},$$
thanks to the comparison principle, 
and that $\bar t\leq n$, one eventually derives
$$\ln \|u_{P} (\tau+Kn,\cdot)\|_{L^{\infty}(\O)} - 
\ln \| u_{P} (\tau,\cdot)\|_{L^\infty(\O)}>
-n\|c\|_\infty- (Kn+n)\mu_{b,p}(P,I)+K.$$
Choosing $K$ large enough, one obtains a contradiction with \eqref{Kn}. 



The claim \eqref{sn} is proved. 
Consider now $I=\R^+$ and the limit operator $P^*_1$ of $P$ associated with 
(a subsequence of) $(s_{n})_{n\in\N}$. 
We have already observed that (up to subsequences)
$$\frac{u_{P}(\cdot+s_{n},\cdot)}{\|u_{P}(s_{n},\cdot)\|_{L^{\infty}(\O)}} \to u_{P^*_1} \quad \hbox{ as } n\to +\infty \hbox{ \ in } \;
W^{1,2}_{p,loc}(\R\times\ol\O).$$
It follows from \eqref{sn} that
$$\forall t\geq 1, \quad \frac{1}{t} \ln \|u_{P^*_1} (t,\cdot)\|_{L^{\infty}(\O)}\leq 
- \mu_{b,p}(P,I)$$
and thus, taking the $\limsup_{t\to +\infty}$ of both sides, we deduce from Theorem \ref{thm:caraceta} that
$\mu_{b,p}(P,\R^+)\leq \lambda_b(P^*_1,\R^+)$.
In an analogous way one finds another limit operator $P^*_2$ such that 
$\lambda_{b,p}(P,\R^+)\geq \mu_b(P^*_2,\R^+)$.
%
The conclusion follows by collecting all the inequalities. 
\end{proof}




%
%
%
%
%
%
%
%
%
%
%
%
%
%
%
%
%
%
%
%
%
%
%
%
%

%
%
%
%
\end{document}